\documentclass[12pt]{article}
\usepackage{amsmath}
   \usepackage{amssymb}
    \usepackage{amsthm}
     \usepackage{latexsym}
      \usepackage{cite}
       \usepackage{verbatim}

\usepackage[colorlinks=true]{hyperref}
 \hypersetup{urlcolor=blue, citecolor=red}
\makeatletter \@addtoreset{equation}{section} \makeatother

\newtheorem{theorem}{Theorem}[section]
\newtheorem{corollary}[theorem]{Corollary}

\newtheorem{lemma}[theorem]{Lemma}
\newtheorem{proposition}[theorem]{Proposition}

\theoremstyle{definition}
\newtheorem{definition}[theorem]{Definition}
\newtheorem{example}[theorem]{Example}
\newtheorem{remark}[theorem]{Remark}
\newtheorem{observation}[theorem]{Observation}

\def\min{{\rm min}}

\def\fq{{\mathbb F}_q}

\newcommand{\lapprox}{\stackrel{<}{_{\sim}}}

\begin{document}
\title{Further results on multiple coverings of the farthest-off points}

\author{}


\maketitle

\centerline{\scshape Daniele Bartoli}
\medskip
{\footnotesize
 \centerline{Department of Mathematics}
 \centerline{Ghent University, Gent, 9000, Belgium}
 }

\bigskip

\centerline{\scshape Alexander A. Davydov }
\medskip
{\footnotesize
 \centerline{Institute for Information Transmission Problems (Kharkevich
 institute) }
 \centerline{Russian Academy of
 Sciences, GSP-4, Moscow, 127994, Russian Federation}
}

\bigskip

\centerline{\scshape Massimo Giulietti, Stefano Marcugini and
Fernanda Pambianco}
\medskip
{\footnotesize
 \centerline{Department of Mathematics and Informatics}
 \centerline{Perugia University, Perugia, 06123, Italy}
 }

\bigskip


\begin{abstract} Multiple coverings of the farthest-off points ($(R,\mu)$-MCF codes) and the corresponding $(\rho,\mu)$-saturating sets in
projective spa\-ces   $PG(N,q)$ are considered. We propose and
develop some methods which allow us to obtain new small
$(1,\mu)$-saturating sets and short
$(2,\mu)$-MCF codes with  $\mu$-density either equal to 1
(optimal saturating sets and almost perfect MCF-codes) or close
to 1 (roughly $1+1/cq$, $c\ge1$). In particular, we  provide
new algebraic constructions and some bounds. Also, we  classify
minimal and optimal $(1,\mu)$-saturating sets in $PG(2,q)$, $q$
small.
\end{abstract}

\section{Introduction}
A code $C$ is called  $(R,\mu)$-{\em multiple covering of the
farthest-off points} (or $(R,\mu)$-MCF for short) if for every
word $x$ at distance $R$ from $C$ there are at least $\mu$
codewords in the Hamming sphere $S(x,R)$, where $R$ is the
covering radius of $C$.

Multiple coverings can be viewed as a generalization of  covering codes,
see \cite{BrLyWi-Handbook,Coh}. Motivations for studying
MCF codes come from the generalized football pool
problem (see e.g. \cite{HHLO-AMM,MR1082845,MR1441669} and the
references therein) and list decoding (see e.g.
\cite{WCL1991}).

In
\cite{Coh,MMCCFF2,HHLO-Siam,MMCCFF5,HonkLits1996,Quistorff,MMCCFF1}
results on MCF codes, mostly concerning the binary and the
ternary cases, can be found. The development of this topic for
arbitrary $q$ was presented in
\cite{BDGMP-ACCT2012,GiulBritCombConf,PDBGM-ExtendAbstr} and in
the recent paper \cite{BDGMP-AMC2015}. In particular, important
parameters of $(R,\mu)$-MCF codes such as the $\mu$-{\em
density} and the $\mu$-length function have been introduced in
\cite{BDGMP-AMC2015}. In the same paper the
notion of a $(\rho,\mu)$-saturating set as the geometrical
counterpart of $(\rho+1,\mu )$-MCF codes was proposed.  Many
useful results and constructions of MCF codes were obtained in
\cite{BDGMP-AMC2015} by geometrical methods. For an
introduction to projective spaces over finite fields see
\cite{Hirs1998}.

The $\mu$-{\em density} of an $(R,\mu)$-MCF code $C$
is the average value of $\frac{1}{\mu}\#(S(x,R)\cap C)$, where
$x$ is a word at distance $R$ from $C$. The $\mu$-{\em density}
is greater than or equal to~1. If the minimum distance $d$ of
$C$ is at least $2R-1$, then the best $\mu$-density among
linear $q$-ary codes with same codimension $r$ and covering
radius $R$ is achieved by the shortest ones. An important class
of MCF codes are \emph{almost perfect} and \emph{perfect} MCF
codes which correspond to \emph{optimal} saturating sets. For
these codes each word at distance $R$ from the code belongs to
{exactly} $\mu $ spheres centered in codewords; they have
the best possible $\mu$-density, i.e. equal to 1. The
$\mu$-length function $\ell _{\mu }(R,r,q)$  is defined as the
smallest length $n$ of a linear $(R,\mu )$-MCF code with
parameters $ [n,n-r,d]_{q}R$, $d\geq 3$.

In this paper, we continue and develop the geometrical
approach of \cite{BDGMP-AMC2015} for constructing MCF codes
with small $\mu$-density. We present a number of
$(1,\mu)$-satura\-ting sets (and the corresponding
$(2,\mu)$-MCF codes) with good parameters.

In the space $PG(N,q)$, $q>2$ even, we obtain
$(1,\mu)$-saturating sets with $\mu=\frac{q-2}{2}$ such that the 
$\mu$-density of the corresponding $(2,\mu)$-MCF code tends to
1 when $N$ is fixed and $q$ tends to infinity, see Section
\ref{q even}.

New results concerning $(1,\mu)$-saturating sets in planes
$PG(2,q)$ are presented in Section
\ref{sezquattro}. We give some upper and lower bounds on the
size of $(1,\mu)$-saturating sets, see Subsection
\ref{subsec_bounds}. Also, we present many examples of optimal
saturating sets  using classical geometrical objects such as 
partitions of $PG(2,q)$ in Singer point-orbits and sets of
convenient lines, see  Sections \ref{sec_classic_object} and
\ref{sec_classific} and Subsection \ref{subsec_Singer}.
Unfortunately, it is not always possible to construct almost
perfect codes; in some cases we construct examples of
$(1,\mu)$-saturating sets with $\mu$-density roughly of the same 
order of magnitude of $1+1/cq$, $c>1$. In general,
we give families of $\mu$-saturating sets of size less than
$\mu \overline{\ell}(2,3,q)$, where $\overline{\ell}(2,3,q)$ is
the minimum known size of a $1$-saturating set in
$PG(2,q)$, see e.g. \cite{BFMP-JG2013,DGMP-AMC2011,GEJC} and
the references therein.

Another achievement of this paper is the
classification of minimal and optimal $(1,\mu)$-saturating sets
in $PG(2,q)$ for small $q$, see Section \ref{sec_classific}.

The paper is organized as follows. In Section
\ref{secDefinCodTheor} we recall some  definitions and results
from  \cite{BDGMP-AMC2015} concerning MCF codes and
$(\rho,\mu)$-saturating sets; in Section \ref{seztre} we
focus on $(1,\mu)$-saturating sets. In Section \ref{q even} we
deal with $(1,\mu)$-saturating sets in $PG(N,q)$, $q$ even,
having small size. In Section \ref{sec_classic_object} perfect
and almost perfect $(2,\mu)$-MCF codes are constructed from classical geometrical objects. In Section
\ref{sezquattro} we present some constructions and bounds on
$(1,\mu)$-saturating sets in $PG(2,q)$. Finally, in
Section~\ref{sec_classific}, computational results on the 
classification of minimal and optimal $(1,\mu)$-saturating sets in $PG(2,q)$ are presented.

\section{Multiple coverings and $(\rho,\mu)$-saturating sets \label{secDefinCodTheor}}
In the following we use the same notation as in
\cite{BDGMP-AMC2015}. An $(n,M,d)_{q}R$  code $C$ is a code of
length $n$, cardinality $M$, minimum distance $d$, and covering
radius $R$, over the finite field ${\mathbb{F}} _{q}$ with $q$
elements.  If $C$ is linear of dimension $k$ over $\fq$, then
$C$ is also said to be an  $[n,k,d]_{q}R$ code. When either $d$
or $R$ are not relevant or unknown they can be omitted in the
above notation.
Let ${\mathbb{F}}_{q}^{n}$ be the linear space of dimension $n$
over ${\mathbb{F}}_{q}$, equipped with the Hamming distance.
The Hamming sphere of radius $j$ centered at $x\in $
${\mathbb{F}}_{q}^{n}$ is denoted by $S(x,j)$. The size
$V_{q}(n,j)$ of such a sphere is
\begin{equation*}
V_{q}(n,j)=\sum_{i=0}^{j}{\binom{n}{{i}}}(q-1)^{i}.
\label{eqDCT_size_sphere}
\end{equation*}
 Let
$\overline{S}(x,R)$ be the surface of the sphere $S(x,R)$. For
an $ (n,M)_{q}R$ code $C$, $ A_{w}(C)$ denotes  the number of
codewords in $C$ of weight $w$, and $f_{\theta }(e,C)$ denotes
the number of codewords at distance $\theta $ from a vector $e$
in ${\mathbb{F}}_{q}^{n}$; equivalently, $f_{\theta }(e,C)=$
$\#(\overline{S}(e,\theta )\cap C).$
 Let
\begin{equation*}
\delta (C,R)=\frac{\sum_{x\in {\mathbb{F}}_{q}^{n}}\#\{c\in C\mid d(c,x)\leq
R\}}{q^{n}}=\frac{M\cdot V_{q}(n,R)}{q^{n}}  \label{eq2_delta_CR}
\end{equation*}
be the \emph{density} of an $(n,M)_{q}R$ covering code $C$. In general,   $\delta (C,R)\geq 1$, and  equality holds if and only if $C$ is a
perfect code.

\begin{definition}[\!\!\cite{BDGMP-AMC2015,Coh,HHLO-AMM,HHLO-Siam}]\label{30may}
$\textrm{ }$\\
\vspace*{-0.4 cm}
\begin{enumerate}
\item An $(n,M)_{q}R$ code $C$ is said to be an \emph{$(R,\mu )$
multiple covering of the farthest-off points} $((R,\mu )$-MCF
code for short) if for all $x\in {\mathbb{F}}_{q}^{n}$ such
that $ d(x,C)=R$ the number of codewords $c$ such that
$d(x,c)=R$ is at least $\mu$.
\item An $(n,M,d(C))_{q}R$ code $C$ is said to be an \emph{$(R,\mu )$
almost perfect multiple covering of the farthest-off points}
$((R,\mu )$-APMCF code for short) if for all $x\in
{\mathbb{F}}_{q}^{n}$ such that $ d(x,C)=R$ the number of
codewords $c$ such that $d(x,c)=R$ is exactl\emph{}y $\mu$. If,
in addition, $d(C)\geq2R$ holds, then the code is called
\emph{$(R,\mu )$  perfect multiple covering of the farthest-off
points} $((R,\mu )$-PMCF code for short).
\end{enumerate}
\end{definition}

In the literature, MCF codes are also called multiple coverings
of deep holes, see e.g. \cite[Chapter 14]{Coh}.

As already pointed out in \cite[Sect. 2]{BDGMP-AMC2015}, there exists a connection between $((R,\mu )$-MCF and $((R,\mu
)$-PMCF with weighted $\mathbf{m}$-coverings; see e.g.
\cite[Sect.~13.1]{Coh}.

\begin{definition}[\!\!\cite{BDGMP-AMC2015}]
Let $C$ be a $((R,\mu )$-MCF code. Let $\{x_1,\ldots,x_{N_{R}(C)}\}$ be the
set of vectors in ${\mathbb{F}}_{q}^{n}$ with distance $R$ from $C$. The
{\em $\mu$-density} of $C$ is
\begin{equation}
\gamma_\mu
(C,R)=\frac{\sum_{i=1}^{N_{R}(C)}f_{R}(x_{i},C)}{\mu N_{R}(C)}.
  \label{eq1_gamma(C,R)isAverage}
\end{equation}
\end{definition}
It is easily seen that $\gamma_\mu (C,R)\ge 1$, and that $C$ is
an $(R,\mu)$ APMCF code precisely when equality holds. In general, this parameter is a measure of the quality of an $((R,\mu )$-MCF code.

The goal of this paper is the construction of $((R,\mu
)$-MCF codes with small $\mu$-density.

We recall the following proposition from
\cite{BDGMP-AMC2015} concerning the $\mu$-density of an $((R,\mu
)$-MCF code.
\begin{proposition}\label{lem1_delta}
Let $C$ be
a linear
$[n,k,d(C)]_{q}R$ code with $d(C)\geq 2R-1$. If $C$ is
 $(R,\mu )$-MCF, then
\begin{equation*}
\gamma_{\mu }(C,R)=\frac{{\binom{n}{{R}}}\cdot
(q-1)^{R}-{\binom{2R-1}{{R-1} }}\cdot A_{2R-1}(C)}{\mu \cdot
(q^{n-k}-V_{q}(n,R-1))}. \label{eq1_mudens_d>=2R-1}
\end{equation*}
\end{proposition}

In the rest of the paper we will  assume that
\begin{equation}
d(C)\geq 2R-1.  \label{eq2_dC>=2R-1}
\end{equation}
Let $ t=\left\lfloor \frac{d-1}{2}\right\rfloor $ be the number
of errors that can be corrected by a code with minimum distance
$d$. Note that under Condition (\ref{eq2_dC>=2R-1}), $R=t+1$;
equivalently, $C$ is a quasi-perfect code in the classical
sense.

The following definition of a $(\rho,\mu)$-saturating set in $PG(N,q)$ is given as in \cite{BDGMP-AMC2015}.

\begin{definition}\label{def_rho_mu_sat}
Let $S=\{P_{1},\ldots ,P_{n}\}$ be a subset of points of
$PG(N,q)$. Then $S$ is said to be $(\rho ,\mu )$-saturating if:

\begin{itemize}
\item[(M1)] $S$ generates $PG(N,q)$;

\item[(M2)] there exists a point $Q$ in $PG(N,q)$ which does not belong to
any subspace of dimension $\rho -1$ generated by the points of $S$;

\item[(M3)] every point $Q$ in $PG(N,q)$ not belonging to any subspace of
dimension $\rho -1$ generated by the points of $S$, is such that the number
of subspaces of dimension $\rho $ generated by the points of $S$ and
containing $Q$, counted with multiplicity, is at least $\mu $. The
multiplicity $m_{T}$ of a subspace $T$ is computed as
the number of distinct sets of $\rho +1$
independent points contained in $T\cap S$.
\end{itemize}
\end{definition}
Note that if any $\rho +1$ points of $S$ are linearly independent (that is,
the minimum distance of the corresponding code is at least $\rho +2$), then
\begin{equation*}
m_{T}={\binom{{\#(T\cap S)}}{{\rho +1}}.}
\end{equation*}

A $(\rho ,\mu )$-saturating $n$-set in $ PG(N,q)$ is called
\emph{minimal }if it does not contain a $(\rho ,\mu
)$-saturating $(n-1)$-set in $PG(N,q)$.

Let $S$ be a $(\rho ,\mu)$-saturating $n$-set in $ PG(n-k-1,q).
$ The set $S$ is called \emph{optimal }$(\rho ,\mu)$-{\em
saturating set }$((\rho ,\mu )$-OS set for short) if every
point $Q$ in $ PG(n-k-1,q)$ not belonging to any subspace of
dimension $\rho -1$ generated by the points of $S$, is such
that the number of subspaces of dimension $\rho $ generated by
the points of $S$ and containing $Q$, counted with
multiplicity, is exactly $\mu $.

An $[n,k]_{q}R$ code $C$ with $R=\rho +1$ {\em
corresponds} to a $(\rho ,\mu )$-saturating $n$-set $S$ in $
PG(n-k-1,q)$ if $C$ admits a parity-check matrix whose columns
are homogeneous coordinates of the points in $S$.


By \cite[Proposition 3.6]{BDGMP-AMC2015}, a linear $[n,k]_{q}R$
code $C$ corresponding to a $(\rho ,\mu )$-saturating $n$-set
$S$ in $PG(n-k-1,q)$ is a $(\rho +1,\mu )$-MCF code. Also, if
$S$ is a $(\rho,\mu)$-OS set, the corresponding linear $[n,
k]_{q}R$ code $C$ is a $(\rho+1,\mu)$ APMCF code with
$\gamma_{\mu}(C,\rho + 1) = 1$. If in addition $d(C)=2R$, then it
is a $(\rho+1,\mu)$ PMCF code.

\section{$(1,\mu)$-saturating sets}\label{seztre}

For $\rho=1$ Conditions (M1)-(M3) read as follows:

\begin{itemize}
\item[(M1)] $S$ generates $PG(N,q)$;

\item[(M2)] $S$ is not the whole $PG(N,q)$;

\item[(M3)] every point $Q$ in $PG(N,q)$ not belonging to $S$ is such that
the number of secants of $S$ through $Q$ is at least $\mu $, counted with
multiplicity. The multiplicity $m_{\ell }$ of a secant $\ell $ is computed
as
\begin{equation*}
m_{\ell }={\binom{{\#(\ell \cap S)}}{{2}}}.
\end{equation*}
\end{itemize}

As already observed in \cite{BDGMP-AMC2015}, from Conditions (M1)-(M3) the following holds.
\begin{proposition}
\label{prop optimal (1,mu) OS}
\begin{itemize}
                    \item[(i)] Let $C$ be the linear
                        $[n,n-N-1]_{q}2$ code corresponding
                        to a $(1,\mu )$-saturating $n$-set
$S$. Then $\mu\gamma_\mu (C,2)$ is equal to the average number
of secants of $S$, counted with multiplicity, through a fixed
point $Q\in PG(N,q)\setminus S$.
\item[(ii)]  Let $S$ be a $(1,\mu )$-saturating set in $
    PG(N,q)$.
    Then $S$ is a $(1,\mu )$-\emph{OS set}
precisely when each point $Q\in PG(N,q)\setminus S$ belongs to
exactly $\mu $ secants of $S$, counted with multiplicity.
\end{itemize}
\end{proposition}


Note that as $R=2$, the condition $d(C)>2R-1$
reads as $ d(C)>3$.

 Let $B_{3}(S)$
denote the number of triples of collinear points in $S$. The
following is a characterization of $(1,\mu )$-\emph{OS sets} in
$ PG(N,q)$; see \cite{BDGMP-AMC2015}.

\begin{proposition}\label{prop optimal (1,mu) OSbis} Let $S$ be a
      $(1,\mu )$-saturating set in $ PG(N,q)$. Let $C_{S}$ be the $[n,n-N-1]_{q}2$ code corresponding to $S$.
\begin{itemize}
  \item [(i)] For the $\mu$-density of $C_{S}$ it holds
      that
\begin{equation}
\gamma_{\mu}(C_{S},2)=\frac{\frac{n-1}{2}(q-1)-\frac{3}{n}B_{3}(S)}{\mu \cdot \left(
\frac{\#PG(N,q)}{n} -1\right)}.  \label{eq2_uniformity density}
\end{equation}
  \item [(ii)] The set $S$ is a $(1,\mu )$-\emph{OS set}
      if and only if
\begin{equation*}
\frac{n-1}{2}(q-1)-\frac{3}{n}B_{3}(S)=\mu \cdot \left(
\frac{\#PG(N,q)}{n} -1\right).  \label{eq2_uniformity}
\end{equation*}
\end{itemize}
\end{proposition}

It is clear that  if $q,N,\mu $ are fixed, then the best $\mu$-density is achieved for
small $n$ and therefore, the following parameter seems to be relevant in this
context.

\begin{definition}
\label{def2} The $\mu $-length function $\ell _{\mu }(2,r,q)$
is the smallest length $n$ of a linear $(2,\mu )$-MCF code with
parameters $ [n,n-r,d]_{q}2$, $d\geq 3$, or equivalently the
smallest cardinality of a $ (1,\mu )$-saturating set in
$PG(r-1,q)$. For $\mu =1$,  $\ell _{1}(2,r,q)$ is the usual
length function $\ell (2,r,q)$
\cite{BrLyWi-Handbook,Coh,DGMP-AMC2011} for 1-fold coverings.
\end{definition}

\begin{remark}\cite{BDGMP-AMC2015}
\label{repeating} A number $\mu $ of disjoint copies of
a $1$ -saturating set in $PG(N,q)$ give rise to a $(1,\mu
)$-saturating set in $ PG(N,q)$. Therefore,
\begin{equation}
\ell _{\mu }(2,r,q)\leq \mu \ell (2,r,q).  \label{disjoint}
\end{equation}
Denote by $\gamma _{\mu }(2,r,q)$ the
minimum $\mu $-density of a linear $(2,\mu )$-MCF code of
codimension $r$ over ${\mathbb{F} }_{q}$. Let $\delta (2,r,q)$
be the minimum density of a linear code with covering radius
$2$ and codimension $r$ over ${\mathbb{F}}_{q}$, then
\begin{equation}
\gamma _{\mu }(2,r,q)\leq \frac{\frac{1}{2}(\mu \ell (2,r,q)-1)(q-1)}{\mu
\cdot (\frac{\#PG(r-1,q)}{\mu \ell (2,r,q)}-1)}-1\sim \mu \delta (2,r,q).
\label{triv1}
\end{equation}
The same inequalities clearly hold for the best known
lengths and densities, denoted, respectively, by
$\overline{\ell }_{\mu }(2,r,q),$ $ \overline{\ell }(2,r,q),$
$\overline{\gamma }_{\mu }(2,r,q)$, and $\overline{ \delta
}(2,r,q)$:
\begin{equation}
\overline{\ell }_{\mu }(2,r,q)\leq \mu \overline{\ell }(2,r,q).
\label{disjoint_known}
\end{equation}
\begin{equation}
\overline{\gamma }_{\mu }(2,r,q)
\lapprox \mu
\overline{\delta }(2,r,q).  \label{triv1_known}
\end{equation}
\end{remark}

From Equations \eqref{disjoint}--\eqref{triv1_known}, results
for parameters $\ell _{\mu }(2,r,q),$ $\overline{\ell }_{\mu
}(2,r,q),$ $\gamma _{\mu }(2,\linebreak r,q),$ and
$\overline{\gamma }_{\mu }(2,r,q),$ can be immediately obtained
from the vast body of literature on $1$-saturating sets in
finite projective spaces; see e.g.
\cite{BFMP-JG2013,Bo-Sz-Ti,BrLyWi-Handbook,Coh,Dav95,DGMP-2010,DGMP-AMC2011,GEJC,GJCD,GiulBritCombConf,GiPa,GT2004,BPFMsatu,TamasRoma}.

The aim of the present paper is to construct
$(1,\mu)$-saturating sets in $PG(N,q)$ giving rise to $(2,\mu
)$-MCF codes with cardinality and density smaller to those in Inequalities
\eqref{disjoint}--\eqref{triv1_known}.

\section{Small $(1,\mu)$-saturating sets in $PG(N,q)$, $q$ even}\label{q even}

In $PG(N,q)$, $q$ even, small $1$-saturating sets have been
constructed; see \cite{DGMP-2010,GJCD}. Roughly speaking, if
$N$ is odd, then in $PG(N,q)$, $q$ even, there are
$1$-saturating sets whose size is of the same order of
magnitude as $ 2q^{(N-1)/2}$. If $N$ is even, then there exist
$1$-saturating sets of size about $t_2(q)q^{(N-2)/2}$, where
$t_2(q)$ denotes the size of the smallest saturating set in
$PG(2,q)$. When $q$ is a square, $t_2(q)\le 3\sqrt q  - 1$
holds \cite{Dav95}.  Therefore, for $q$ a square,
\begin{equation*}
\overline{\ell}(2,r,q)\sim c(r)q^{(r-2)/2}, \text{ with } c(r)=\left\{
\begin{array}{ll}
2 & \text{ for even } r \\
3 & \text{ for odd } r
\end{array}
\right.,
\end{equation*}
and the following results on density  about the order of
magnitude of $\overline{\gamma}_{\mu}(2,r,q)$ can be easily
obtained, see \eqref{triv1_known}:
\begin{equation}  \label{triveven}
\overline{\gamma}_{\mu}(2,r,q)\sim \frac{1}{2}\mu c(r)^2.
\end{equation}

In this section we significantly improve \eqref{triveven} for
the case $\protect\mu=\frac{q-2}{2}$, see \eqref{trasl1}
and~\eqref{trasl2}.

For $i=0,\ldots, N$, let $\pi_i$ be the subset of $PG(N,q)$
defined as follows:
\begin{equation*}
\pi_i:=\{(x_0,\ldots,x_N)\mid x_0=x_1=\ldots=x_{i-1}=0, x_i\neq 0\},
\end{equation*}
where $x_0,\ldots,x_N$ are the homogeneous coordinates of a
point in $PG(N,q)$. Clearly, $PG(N,q)$ is the disjoint union of
$\pi_0\cup\pi_1\cup\ldots\cup \pi_N $; also, each $\pi_i$ with
$i<N$ can be viewed as an affine space $ AG(N-i,q)$, whereas
$\pi_N$ consists of a single point, namely $ (0,\ldots,0,1)$.

\begin{lemma}
\label{5gen} In an affine space $AG(N,q)$ with $q$ even, $q>2$, there exists
a subset $K$ of size less than or equal to
\begin{equation*}
\left\lceil \frac{1+\sqrt{4q-7}}{2}\right\rceil q^{(N-1)/2}
\end{equation*}
such that every point of $AG(N,q)\setminus K$ belongs to at least $(q-2)/2$
distinct secants of $K$.
\end{lemma}

\begin{proof}
Assume that $N$ is even. Then there exists a translation cap $K$ in $AG(N,q)$ of size $q^{N/2}$ (see (2.6) in \cite{GJCD}). By Proposition 2.5 in \cite{GJCD},
every point of $AG(N,q)\setminus K$ belongs to $(q-2)/2$ distinct secants of $K$. Then the assertion follows from
$$
\sqrt q\le \left\lceil  \frac{1+\sqrt{4q-7}}{2}\right\rceil.
$$

To deal with the case of odd dimension $N$, we set $s=\left\lceil \frac{1+\sqrt{4q-7}}{2}\right\rceil$ and we fix $s$ distinct elements $a_1,\ldots,a_s$ in $\fq$.
Note that
$$
{s \choose 2}  \ge \frac{q-2}{2}.
$$
Let $K'$ be a translation cap  in $AG(N-1,q)$ of size $q^{(N-1)/2}$. Let
$$
K=\{(P,a_i)\mid P\in K', i=1,\ldots,s\}\subset AG(N,q).
$$
By the  doubling construction (see e.g. Remark 2.14 in \cite{GJCD}), we have that for each pair of integers $i,j$ with $1\le i<j\le s$, the subset of $K$
$$
K_{i,j}=\{(P,a_i),(P,a_j)\mid P\in K'\}
$$
is a complete cap of size $q^{(N-1)/2}$ in $AG(N,q)$.

Therefore, every point of $AG(N,q)\setminus K$ is covered by at least ${s\choose 2}\ge (q-2)/2$ secants of $K$.
\end{proof}

A slight improvement of Lemma \ref{5gen} can be obtained when $q$ is a
square and there exists a translation cap of size $q^{3/2}$ in $AG(3,q)$.

\begin{lemma}
\label{5genbis} Assume that $q$ is a square and there exists a
translation cap of size $q^{3/2}$ in $AG(3,q)$. Then in an
affine space $AG(N,q)$ with $ q $ even, $N>2$, there exists a
subset $K$ of size less than or equal to $ q^{N/2}$ such that
every point of $AG(N,q)\setminus K$ belongs to at least $
(q-2)/2$ distinct secants of $K$.
\end{lemma}

\begin{proof}
The assertion for $N$ even follows from the proof of Lemma \ref{5gen}.
Here it is possible to construct a translation cap in $AG(N,q)$ of size $q^{N/2}$ also for odd $N>1$ by using Proposition 2.8 in \cite{GJCD}.
\end{proof}

\begin{remark}
\label{rem5gen} The hypothesis of Lemma \ref{5genbis} are satisfied for
instance for $q=16$ (see \cite[Lemma 3.1]{GiPa}).
\end{remark}

Consider the partition
\begin{equation*}
PG(N,q)=\pi_0\cup \pi_2\cup\ldots \cup \pi_{N-1}\cup \{(0,\ldots,0,1)\}.
\end{equation*}
As each $\pi_i$ is an $AG(N-i,q)$, by Lemma \ref{5gen} we are
able to find sets $K_i$ contained in $\pi_i$, of size at most
\begin{equation*}
\left\lceil \frac{1+\sqrt{4q-7}}{2}\right\rceil q^{(N-i-1)/2},
\end{equation*}
and such that each point of $\pi_i\setminus K_i$ belongs at least $(q-2)/2$
distinct secants of $K_i$.

Then the following result is obtained by considering the union
of the $K_i$ 's, together with the point \{(0,\ldots,0,1)\}.

\begin{theorem}
\label{befana} In a projective space $PG(N,q)$ with $q$ even,
$q>2$, there exists a subset $S$ of size equal to
\begin{equation*}
1+\left\lceil \frac{1+\sqrt{4q-7}}{2}\right\rceil (q^{\frac{N-1}{2}}+q^{
\frac{N-2}{2}}+\ldots+q^{\frac{1}{2}}+1)
\end{equation*}
such that every point of $PG(N,q)\setminus S$ belongs to at
least $(q-2)/2$ distinct secants of~$S$.

If in addition $q$ is a square and there exists a translation
cap of size $ q^{3/2}$ in $AG(3,q)$, then $S$ can be chosen in
such a way that
\begin{equation*}
|S|= q^{\frac{N}{2}}+q^{\frac{N-1}{2}}+q^{\frac{N-2}{2}}+\ldots+q+q^{\frac{1
}{2}}+1=\frac{q^{(N+1)/2}-1}{q^{1/2}-1}.
\end{equation*}
\end{theorem}

 As for the density, we have that if
$C(N,q)$ is the code corresponding to the set $S$ of Theorem
\ref{befana}, then by Proposition \ref{prop optimal (1,mu)
OSbis}(i),
\begin{equation*}
\gamma_{\frac{q-2}{2}}(C(N,q),2)\le\frac{\frac{|S|-1}{2}(q-1)}{\frac{q-2}{2}(\frac{
|PG(N,q)|}{|S|}-1)}=\frac{q-1}{q-2} \cdot \frac{|S|(|S|-1)}{|PG(N,q)|-|S|},
\end{equation*}
whence
\begin{equation}  \label{trasl1}
\gamma_{\frac{q-2}{2}}(C(N,q),2)<\frac{q-1}{q-2}\cdot \frac{q^N+10q^{N-\frac{1}{2}
}+25q^{N-1}} {q^{N}+q^{N-1}+\ldots+q+1-q^{\frac{N}{2}}-5q^{\frac{N-1}{2}}}.
\end{equation}
Interestingly, if we fix codimension $N$ and let $q$ vary, we
get that
\begin{equation*}
\lim_{q\to \infty}\gamma_{\frac{q-2}{2}}(C(N,q),2)=1.
\end{equation*}

The situation is even more interesting if the hypothesis of
Lemma \ref {5genbis} are satisfied. In this case we obtain the
following very nice formula for
$\gamma_{\frac{q-2}{2}}(C(N,q),2)$, which is independent
of $N$:
\begin{equation}  \label{trasl2}
\gamma_{\frac{q-2}{2}}(C(N,q),2)=\frac{q-1}{q-2}\cdot\frac{\frac{q^{(N+1)/2}-1}{
q^{1/2}-1}-1}{\frac{q^{(N+1)/2}+1}{q^{1/2}+1}-1}=\frac{q-1}{q-2}\cdot \frac{
\sqrt q+1}{\sqrt q-1}=\frac{(\sqrt q+1)^2}{q-2}.
\end{equation}

For instance for $q=16$ we have $|S|=\frac{4^{N+1}-1}{3}$, and
we get that $\gamma_{\frac{q-2}{2}}(C(N,q),2)$ is equal to
$25/14$ independently of $N$.


\section{Perfect and almost perfect $(2,\mu)$-MCF codes from
classical geometrical objects in
$PG(N,q)$}\label{sec_classic_object} In this section we
construct optimal $ (1,\mu )$-saturating $n$-sets ($ (1,\mu
)$-OS $n$-sets) in $PG(N,q)$. Recall that an $[n,n-(N+1),d(C)]$
code $C$ corresponding to a $ (1,\mu )$-OS set is an almost
perfect $(2,\mu)$-MCF code (APMCF code) if $d(C)=3$ or perfect
$(2,\mu)$-MCF code (PMCF code) if $d(C)=4$, see Sections
\ref{secDefinCodTheor} and \ref{seztre}. The $\mu$-density of any APMCF or  PMCF code $C$ is $\gamma _{\mu }(C,2)=1$.

In Proposition \ref{th_oval} we obtain  MCF codes $C$ with
 $\mu$-density $\gamma _{\mu }(C,2)=1+\frac{1}{q}$.

\begin{proposition}
\label{th7_(n,s)arc}Let $q=2^{v}$ be even. Let $s=2^{k}$, $1\leq k\leq v.$
 Finally, let $n=(s-1)q+s$ and let
$\mathcal{K}$ be a maximal $(n,s)$-arc in $PG(2,q)$.  Then $\mathcal{K}$ is a $ (1,\mu )$-OS $n$-set with
      parameters
\begin{equation*}
n=(s-1)q+s,~~\mu =\frac{1}{2}(s-1)n.
\end{equation*}
An $[n,n-3,d(C_{\mathcal{K}})]_{q}2$ code $C_{\mathcal{K}}$
corresponding to $\mathcal{K}$\ is a $\left( 2,\mu\right)
$-PMCF code if $s=2$ and a $\left( 2,\mu\right)$-APMCF code
 if $s\geq 4.$

\end{proposition}

\proof
Every line of $PG(2,q)$ meets a maximal
            $(n,s)$-arc
either in zero or in $s$ points whence it follows that
every point of $PG(2,q)$ outside the arc lies on
$\frac{n}{s}$ $s$-secant. Therefore, $R=2$ and $\mu
=\binom{s }{2}\frac{n}{s}=\frac{1}{2}(s-1)n.$  For $s=2$
the minimum distance $d(C_{\mathcal{K}})$ is $4$ as in this case the arc $\mathcal{K}$ is a hyperoval.

\endproof

\begin{proposition}
\label{th_elliptic_quadric}An elliptic quadric $\mathcal{Q}$ in
$PG(3,q)$ is a $(1,\mu )$-OS $n$-set with
\begin{equation*}
n=q^{2}+1,~~\mu = \frac{1}{2}(q^{2}-q).
\end{equation*}
A $[n,n-4,4]_{q}2$ code $C_{\mathcal{Q}}$ corresponding to
$\mathcal{Q}$ is a $(2, \mu)$-PMCF code.
\end{proposition}
\proof Every point outside of the elliptic quadric
$\mathcal{Q}$ lies on $ \frac{1}{2}(q^{2}-q)$ bisecants.
\endproof

\begin{proposition}
\label{th_Hermit}Let $q$ be square. A Hermitian curve
$\mathcal{H}$ in $ PG(2,q)$ is a $(1,\mu )$-OS $n$-set with
parameters
\begin{equation*}
n=q\sqrt{q}+1,~~\mu
=\frac{1}{2}(q^{2}-q).
\end{equation*}   An $[n,n
-3,3]_{q}2$ code $C_{\mathcal{H}}$ corresponding to $
\mathcal{H}$ is a $(2,\mu)$-APMCF code.
\end{proposition}

\proof
 Every point outside of the Hermitian curve lies
on $q-\sqrt{q}$ lines that are $(\sqrt{q}+1)$-secants and on $\sqrt{q}+1$ tangent lines to $\mathcal{H}$. Therefore $\mu =(q-\sqrt{q})\binom{
\sqrt{q}+1}{2}=\frac{1}{2}(q^{2}-q)$.
\endproof

\begin{proposition}
\label{th_BaerSubplane}Let $q$ be square. A Baer subplane $\mathcal{B}$  in $ PG(2,q)$  is a $(1,\mu
      )$-OS $n$-set with parameters
      \begin{equation*}
n=q+\sqrt{q}+1,~~\mu =\frac{1}{2}(q+\sqrt{q}).
\end{equation*}
       An $[n,n-3,3]_{q}2$ code $C_{\mathcal{B}}$
      corresponding to $\mathcal{B}$ is a $(2,\mu)$-APMCF
      code.

\end{proposition}

\proof
 Here $\mu
      =\binom{\sqrt{q}+1}{2}=\frac{1}{2}(q+\sqrt{q})$ as
      every point outside the Baer subplane lies exactly on
      one $(\sqrt{q}+1)$-secant of the subplane and on $q$
      tangents to $\mathcal{B}$.
\endproof

\begin{proposition}\label{TheoremComplement}
Let $\mathcal{S}\subset PG(N,q)$ be a set such that through each point $P \in \mathcal{S}$ the number of $i$-secants of $\mathcal{S}$ is a fixed integer $x_i$. Then $PG(N,q)\setminus \mathcal{S}$ is a $(1,\mu)$-OS $n$-set with 

$$n=\frac{q^{N+1}-1}{q-1}-|\mathcal{S}|,~~\mu = \sum_{i=1}^{\frac{q^{N}-1}{q-1}} x_i \binom{q+1-i}{2}.$$
\end{proposition}
\proof
It is enough to observe that each $i$-secant of $\mathcal{S}$  becomes a $(q+1-i)$-secant of $PG(N,q)\setminus \mathcal{S}$. 
\endproof

\begin{corollary}\label{Corollary}
\begin{enumerate}
  \item [(i)]Let $q=2^{v}$ be even and $s=2^{k}$, $1\leq k\leq v-1.$ Consider $n=(s-1)q+s$ and let
$\mathcal{K}$ be a maximal $(n,s)$-arc in $PG(2,q)$. The set
      $\mathcal{S}=PG(2,q)\setminus \mathcal{K}$ is a $
      (1,\mu )$-OS $n$-set with
\begin{equation*}
n=q^{2}+q+1-(s-1)q-s,~~\mu =(q+1)\binom{q+1-s}{2}.
\end{equation*}
An $[n,n-3,3]_{q}2$ code $C_{\mathcal{S}}$ corresponding to
$\mathcal{S}$\ is a $\left( 2,\mu\right) $-APMCF code.
\item [(ii)] Let $q$ be square and let $\mathcal{H}$ be a Hermitian curve in $ PG(2,q)$. The set $\mathcal{S}=PG(2,q)\setminus
      \mathcal{H}$ is a $ (1,\mu )$-OS $n$-set with
      parameters
      \begin{equation*}
n=q^{2}-q\sqrt{q}+q,~~\mu=\binom{q}{2}+q\binom{q-\sqrt{q}}{2}.
\end{equation*}
        An $[n,n-3,3]_{q}2$ code $C_{\mathcal{S}}$
      corresponding to $\mathcal{S}$ is a $(2,\mu)$-APMCF
      code.
  \item [(iii)] Let $q$ be square and consider a 
a Baer subplane  $\mathcal{B}$  in $ PG(2,q)$. The set $\mathcal{S}=PG(2,q)\setminus
      \mathcal{B}$ is a $ (1,\mu )$-OS $n$-set with
      parameters
      \begin{equation*}
n=q^{2}-\sqrt{q},~~\mu=(\sqrt{q}+1)\binom {q-\sqrt{q}}{2}+(q-\sqrt{q})\binom{q}{2}.
\end{equation*}
        An $[n,n-3,3]_{q}2$ code $C_{\mathcal{S}}$
      corresponding to $\mathcal{S}$ is a $(2,\mu)$-APMCF
      code.
\item [(iv)] \label{th_arc}Let $\mathcal{S}\subset PG(N,q)$ be a set such that
$PG(N,q)\setminus \mathcal{S}$ is a
           $k$-cap, $N\ge2$, $k\ge2$. Then $\mathcal{S}$ is a $(1,\mu )$-OS $n$-set with
           parameters
           \begin{equation*}
n=\frac{q^{N+1}-1}{q-1}-k,~~\mu=(k-1)\binom{q-1}{2}+\left(\frac{q^{N}-1}{q-1}-k+1\right)\binom{q}{2}.
\end{equation*}
            An $[n,n-N-1,3]_{q}2$ code $C_{\mathcal{S}}$ corresponding to
$\mathcal{S}$\ is a $\left( 2,\mu\right) $-APMCF code.

\end{enumerate}
\end{corollary}
\proof
The claims follow directly from Proposition \ref{TheoremComplement} and properties of the geometrical objects as pointed out in the previous propositions.
\endproof

\begin{proposition}
\label{th_oval}Let $q$ be odd. An oval $\mathcal{O}$ in
$PG(2,q)$ is a $ (1,\mu )$-saturating $n$-set with parameters
\begin{equation*}
n=q+1,~~\mu =\frac{1}{2}(q-1).
\end{equation*}
An $[n,n-3,4]_{q}2$ code $C_{\mathcal{O}}$ corresponding to
$\mathcal{O}$ is a $(2,\mu)$-MCF code with $\mu $-density
$\gamma_{\mu }(C_{\mathcal{O}},2)=1+\frac{1}{q}.$
\end{proposition}

\proof Each internal point of the oval lies on $\frac{1}{2} (q+1)$ bisecants, whereas every external point lies on $\frac{1}{2}(q-1)$ bisecants. Therefore, $\mu =\frac{1}{2}(q-1).$ By Proposition \ref
{prop optimal (1,mu) OS}(i),
\begin{equation*}
\mu\gamma_{\mu}(C_{\mathcal{O}},2)=\frac{\frac{1}{2}q(q-1)\cdot \frac{1}{2}(q+1)+
\frac{1}{2}q(q+1)\cdot \frac{1}{2}(q-1)}{\frac{1}{2}q(q-1)+\frac{1}{2}q(q+1)}
=\frac{(q+1)(q-1)}{2q}.
\end{equation*}
\endproof

\section{Constructions of small $(1,\mu)$-saturating sets in $PG(2,q)$}\label{sezquattro}
In this section, we summarize some results concerning $(1,\mu)$-saturating sets in projective planes $PG(2,q)$.

\subsection{Bounds}\label{subsec_bounds}
In this subsection we present some upper and lower bounds on
the size of minimal $(1,\mu)$-saturating sets in $PG(2,q)$.
\begin{proposition} \label{prop new_bounds}
In $PG(2,q)$, for a minimal $(1,\mu)$-saturated set $S$ the
following holds.
\begin{itemize}
  \item [(i)]
  \begin{align*}
  \mu\le (q+1)\binom{q}{2}.
  \end{align*}
  \item [(ii)]
  \begin{align*}
  |S|\le\left\{\begin{array}{ccc}
                 q+\mu+1 & \mbox{if} & \mu\le q+2 \\
                 \min\{q+\mu,~q^{2}+q\} &  \mbox{if} & \mu\ge q+3
               \end{array}
  \right. .
  \end{align*}
\end{itemize}
\end{proposition}
\proof
\begin{itemize}
  \item [(i)]Any $(q^{2}+q)$-set in $PG(2,q)$ is a $\left((q+1)\binom{q}{2}\right)$-saturating set.
  \item [(ii)] Let $S$ be $(q+\mu +1)$-set in $PG(2,q)$, $q>2,$ $ \mu \le q+2$ and consider a point  $P\in PG(2,q)\setminus S$. On the $q+1$ lines through $P$ there are at least $ \mu $ pairs of points of $S$ and therefore $S$ is a
      $(1,\mu )$-saturating set, possibly not minimal.

      If $ \mu \ge q+3$ and $|S|=q+\mu$, then at least one
      triple of points of $S$ lies on the same line through
      $P$. So, in total there are at least
      $\binom{3}{2}+(\mu-3)= \mu $ pairs of points of $S$.
      The bound $|S|\le q^{2}+q$ holds due to Condition
      (M2).
\end{itemize}
 \endproof

We recall the following proposition from~\cite{BDGMP-AMC2015} concerning bounds on the smallest possible size of $(1,\mu)$-saturating sets.

\begin{proposition}
\label{lem_lower_bounds} For the length function $\ell _{\mu }(2,3,q)$, the following relations hold.
\begin{itemize}
\item[\textbf{(i)}] Trivial bound:
\begin{equation}
\ell _{\mu }(2,3,q)\geq \sqrt{2\mu q}.  \label{eq4_triv_lower_bound}
\end{equation}

\item[\textbf{(ii)}] Probabilistic bound:
\begin{equation}
\ell _{\mu }(2,3,q) < 66\sqrt{\mu q \ln q}, \quad \textrm{if } \mu <121 q \log q.
\label{eq4_probab_bound}
\end{equation}

\item[\textbf{(iii)}] Baer bound for $q$ a square:
\begin{equation*}
\ell _{\mu }(2,3,q)\leq \mu (3\sqrt{q}-1).
\label{eq4_bound_q_square}
\end{equation*}
\end{itemize}
\end{proposition}


In some cases we can do better than the trivial lower bound
mentioned above.
\begin{proposition}
\label{th:minimalsize} Let $A$ be a $(1,\mu )$-saturating set in $PG(2,q)$ of size $k$.
Suppose that $\ell$ and $\ell^{\prime}$ are an $r$-secant and an $s$-secant of $A$ respectively, with $s\geq r$. Then
$$
 \\[-2 mm]
k\geq \min \left\{ \begin{array}{l}
r+\frac{1}{2} + \sqrt{(s-r)(s+r-2)+2\mu (q-r+1) +\frac{5}{4}},
\\
 r +\frac{1}{2} + \sqrt{(s-r)(s+r-1)+2\mu (q-r) +\frac{1}{4}}
\end{array}
\right\}.$$

\end{proposition}

\proof
There are  $k-r-s+1$ or $k-r-s$ points of $A$ not contained in
$\ell \cup \ell^{\prime}$, depending on the point $\ell \cap
\ell^{\prime}$ belonging or not to $A$. Since $A$ is a $(1,\mu
)$-saturating set, then each point of $\ell$ not belonging to
$A$ is covered at least $\mu$ times. Therefore, in the first
case we obtain
$$(k-r-s+1)(s-1) +\binom{k-r-s+1}{2}+\binom{r}{2}\geq \mu
(q+1-r)$$
which implies
$$k\geq r+\frac{1}{2} + \sqrt{(s-r)(s+r-2)+2\mu (q-r+1) +\frac{5}{4}},$$ whereas in the second case
$$(k-r-s)s +\binom{k-r-s}{2}+\binom{r}{2}\geq \mu (q-r)$$
implies
$$k\geq r +\frac{1}{2} + \sqrt{(s-r)(s+r-1)+2\mu (q-r)
+\frac{1}{4}}.$$ Note that in the second case we do not
consider how many times  the intersection point of $\ell$ and $\ell^{\prime}$ is covered.
\endproof

\subsection{Constructions}\label{subsec_constr}
In this subsection we explicitly construct examples of
$(1,2)$-saturating sets in $PG(2,q)$ of sizes $q+2$
and $q+3$.

\begin{theorem}
\label{th:qp3primo} There exists a minimal $(1,2)$-saturating
set of size $q+3$ in $PG(2,q)$, with $q=p^h$, $p$ prime. Its
stabilizer in $P\varGamma L(3,q)$ has size $hq(q-1)$. The
corresponding $[q+3,q]_{q}2$ code $C$ is a $(2,2)$-MCF with
$\mu$-density $\gamma_{2}(C,2)\thickapprox 1+\frac{1}{2q}$.
\end{theorem}

\proof Let $\ell$ be a line and consider two points
$P_1,P_2\notin \ell$. It is straightforward to check that $A=\ell\cup
\{P_{1},P_{2}\}$ is a minimal $(1,2)$-saturating set. Consider
now two of such sets $A_{1}=\ell _{1}\cup \{P_{1},Q_{1}\}$ and
$ A_{2}=\ell _{2}\cup \{P_{2},Q_{2}\}$, with $P_{i},Q_{i}
\not\in \ell_i$ and let $X_{i},Y_{i}\in \ell _{i}$. The two
sets $A_{1}$ and $A_{2}$  are projective equivalent since
$\varphi :PG(2,q)\rightarrow PG(2,q)$ such that $\varphi
(P_{1})=P_{2}$, $\varphi (Q_{1})=Q_{2}$, $\varphi
(X_{1})=X_{2}$ and $ \varphi (Y_{1})=Y_{2}$ is a collineation
sending $A_{1}$ in $A_{2}$.

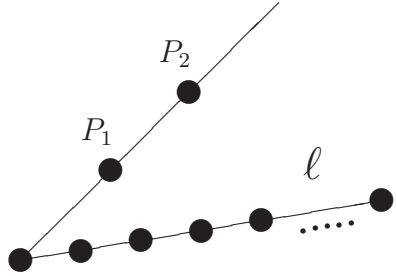
\begin{center}
\setlength{\unitlength}{0.8cm}
\begin{picture}(6,6)
\put(5.7,2.35){\Large $\ell$}
\put(1,1){\line (6,1){6}}
\put(1,1){\line (1,1){4.3}}
\put(1,1){\circle* {0.4}}
\put(2,1.1666){\circle* {0.4}}
\put(3,1.333){\circle* {0.4}}
\put(4,1.5){\circle* {0.4}}
\put(5,1.666){\circle* {0.4}}
\put(5.7,1.5){\circle* {0.08}}
\put(5.9,1.533){\circle* {0.08}}
\put(6.1,1.566){\circle* {0.08}}
\put(6.3,1.6){\circle* {0.08}}
\put(6.5,1.633){\circle* {0.08}}
\put(7,1.999){\circle* {0.4}}
\put(2.5,2.5){\circle* {0.4}}
\put(2,3){$P_{1}$}
\put(3.8,3.8){\circle* {0.4}}
\put(3.3,4.3){$P_{2}$}
\end{picture}
\end{center}

The line $\ell$ can be chosen in $q^2+q+1$ different ways. $P_{1}$ and $P_{2}$ can be taken in an arbitrary way in $PG(2,q)\setminus \ell$. Therefore $A$ can be chosen in $q^2(q^2-1)(q^2+q+1)$. Hence, its stabilizer in $PGL(3,q)$ has size

\begin{equation*}
\frac{|P\varGamma L(3,q)|}{hq^{2}(q^{2}-1)(q^{2}+q+1)}=\frac{
hq^{3}(q^{2}-1)(q^{3}-1)}{q^{2}(q^{2}-1)(q^{2}+q+1)}=hq(q-1).
\end{equation*}

The $\mu$-density of the code $C$ can be calculated by
\eqref{eq2_uniformity density} where clearly $n=q+3$,
$B_{3}(S)=1+\binom{q+1}{3}$, $\#PG(N,q)=q^{2}+q+1$.
\endproof

\begin{theorem}
\label{th:qp3secondo} Let $\ell$ be a line and $P,Q,R,S,T$
points such that $P,R,S$ and $Q,R,T$ are collinear, and $P,Q\in
\ell$. Then $A=(\ell \setminus \{P,Q\}) \cup \{R,S,T\}\subset
PG(2,q)$ is a minimal $(1,2)$-saturating $(q+2)$-set for all
$q\geq 4$.
\end{theorem}

\proof
All the points on the line $PR$ are covered
once by $RS$ and once by the lines joining $T$ and the points of $\ell \setminus \{P\}$. The points of the line $PT$ are covered twice by the lines
through $R$ and $S$ and the the points of $\ell \setminus \{P\}$. A similar argument holds for the points on the lines $QR$ and $QS$. The points
on the lines through $P$ distinct  from $\ell$, $PR$, and $PT$ are covered at
least two times by the lines through $T,R,S$ and the points of $\ell
\setminus \{P\}$. A similar argument holds for the points on the lines through $Q$ distinct from $\ell$, $QR$, and $QS$.

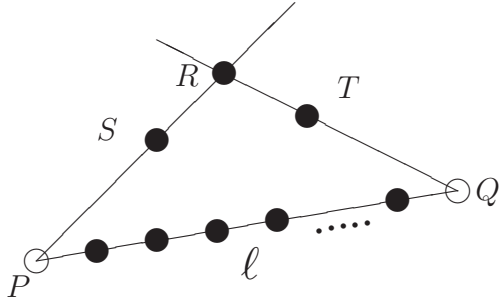
\begin{center}
\setlength{\unitlength}{0.8 cm}
\begin{picture}(6,6)
\put(4.4,0.7){\Large $\ell$}
\put(1,1){\line (6,1){7}}
\put(1,1){\line (1,1){4.3}}
\put(1,1){\circle {0.4}}
\put(0.5,0.4){$P$}
\put(2,1.1666){\circle* {0.4}}
\put(3,1.333){\circle* {0.4}}
\put(4,1.5){\circle* {0.4}}
\put(5,1.666){\circle* {0.4}}
\put(5.7,1.5){\circle* {0.08}}
\put(5.9,1.533){\circle* {0.08}}
\put(6.1,1.566){\circle* {0.08}}
\put(6.3,1.6){\circle* {0.08}}
\put(6.5,1.633){\circle* {0.08}}
\put(7,2){\circle* {0.4}}
\put(8,2.166){\circle {0.4}}
\put(8.3,2){$Q$}
\put(8,2.166){\line(-2,1) {5}}
\put(3,3){\circle* {0.4}}
\put(2,3){$S$}
\put(4.12,4.12){\circle* {0.4}}
\put(3.3,3.9){$R$}
\put(5.5,3.4){\circle* {0.4}}
\put(6,3.7){$T$}
\end{picture}
\end{center}

This example is minimal. In fact it is not possible to delete $T$ (resp. $S$) since the points on the line $SR$ (resp. $RT$) would not be covered twice. $A\setminus \{R\}$ does not cover $R$. Let $X\in \ell \setminus\{P,Q\}$, then $A\setminus \{X\}$ does not cover twice the point $TX\cap m$.
\endproof


\begin{theorem}
\label{th:qp2primo} Let $\ell$ be a line and $P,Q,R,S,T$ points
such that $P,R,S,T$ are collinear, and $P,Q\in \ell$. Then
$A=(\ell \setminus \{P,Q\}) \cup \{R,S,T\}$ is a minimal
$(1,2)$-saturating $(q+2)$-set for all $q\geq 4$.
\end{theorem}

\proof Let $s$ be a line through $Q$ different from $\ell$ and
let $m$ be the line containing $R,S,T$.

Let $X=m\cap s$. If $X\in \{R,S,T\}$, then every point of $s \setminus \{Q\}$ is covered twice by the lines through the points $ \{R,S,T\} \setminus \{X\}$ and the the $q-1$ points of $\ell$. If $X\notin \{R,S,T\}$, then every point of $s \setminus \{Q,X\}$ is covered three times by the lines through $R,S,T$ and the $q-1$ points of $ \ell$.

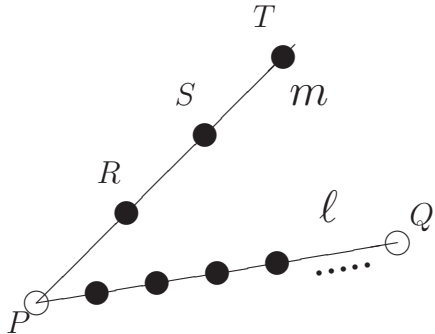
\begin{center}
\setlength{\unitlength}{0.8cm}
\begin{picture}(6,6)
\put(5.7,2.35){\Large $\ell$}
\put(5.2,4.3){\Large $m$}
\put(1,1){\line (6,1){6}}
\put(1,1){\line (1,1){4.3}}
\put(1,1){\circle {0.4}}
\put(0.5,0.5){$P$}
\put(2,1.1666){\circle* {0.4}}
\put(3,1.333){\circle* {0.4}}
\put(4,1.5){\circle* {0.4}}
\put(5,1.666){\circle* {0.4}}
\put(5.7,1.5){\circle* {0.08}}
\put(5.9,1.533){\circle* {0.08}}
\put(6.1,1.566){\circle* {0.08}}
\put(6.3,1.6){\circle* {0.08}}
\put(6.5,1.633){\circle* {0.08}}
\put(7,2){\circle {0.4}}
\put(7.2,2.3){$Q$}
\put(2.5,2.5){\circle* {0.4}}
\put(2,3){$R$}
\put(3.8,3.8){\circle* {0.4}}
\put(3.3,4.3){$S$}
\put(5.1,5.1){\circle* {0.4}}
\put(4.6,5.6){$T$}
\end{picture}
\end{center}

It is not possible to delete $R,S,T$ since in this case the points on $m$ are covered only once. Also, it is not possible to delete a point $X\in \ell$, since in this case in the line $XT$ only $(q-2)$ points are covered twice. Hence $A$ is a minimal $(1,2)$-saturating set of size $(q+2)$.
\endproof

\subsection{$(1,\mu)$-saturating sets and partitions of $PG(2,q)$ in Singer point-orbits}\label{subsec_Singer} 

In \cite{BDGMP-AMC2015,DGMP-Petersb2009,DFGMP-ConfigParam,DGMP-GC},
partitions of $PG(2,q)$ by Singer subgroups are considered.
Methods of \cite{DFGMP-ConfigParam,DGMP-GC}, allow us to
represent an incidence matrix of the plane $PG(2,q)$ as a BDC\
matrix defined below. We present some new results; see
\cite[Sec. 7.3]{DGMP-GC} for comparison. We recall the
following definition.
\begin{definition}
\label{def8_block double-circul}\cite{DFGMP-ConfigParam} Let
$v=td.$ A $ v\times v$ matrix $\mathbf{A}$ is said to be
\emph{block double-circulant} \emph{matrix }(or \emph{BDC
matrix}) if
\begin{equation}
\mathbf{A}=\left[ \renewcommand{\arraystretch}{0.85}
\begin{array}{cccc}
\mathbf{C}_{0,0} & \mathbf{C}_{0,1} & \ldots & \mathbf{C}_{0,t-1} \\
\mathbf{C}_{1,0} & \mathbf{C}_{1,1} & \ldots & \mathbf{C}_{1,t-1} \\
\vdots & \vdots & \vdots & \vdots \\
\mathbf{C}_{t-1,0} & \mathbf{C}_{t-1,1} & \ldots & \mathbf{C}_{t-1,t-1}
\end{array}
\right] ,\text{ }  \label{eq8 _block-circulant}
\end{equation}
\begin{equation}
\mathbf{W(A)}=\left[ \renewcommand{\arraystretch}{0.80}
\begin{array}{ccccccc}
w_{0} & w_{1} & w_{2} & w_{3} & \ldots & w_{t-2} & w_{t-1} \\
w_{t-1} & w_{0} & w_{1} & w_{2} & \ldots & w_{t-3} & w_{t-2} \\
w_{t-2} & w_{t-1} & w_{0} & w_{1} & \ldots & w_{t-4} & w_{t-3} \\
\vdots & \vdots & \vdots & \vdots & \vdots & \vdots & \vdots \\
w_{1} & w_{2} & w_{3} & w_{4} & \ldots & w_{t-1} & w_{0}
\end{array}
\right] ,  \label{eq8_block-circulant-weights}
\end{equation}
where $\mathbf{C}_{i,j}$ is a \emph{circulant} $d\times d$
0,1-matrix for all $i,j$; submatrices $\mathbf{C}_{i,j}$ and
$\mathbf{C}_{l,m}$ with $ j-i\equiv m-l\pmod t$ have equal
weights; $\mathbf{W(A)}$ is a \emph{circulant} $t\times t$
matrix whose entry in a position $i,j$ is the \emph{weight} of
$\mathbf{C}_{i,j}$. $\mathbf{W(A)}$ is called a \emph{weight
matrix} of $\mathbf{A.}$ The vector
$\overline{\mathbf{W}}\mathbf{(A)=(} w_{0},w_{1},\ldots
,w_{t-1}\mathbf{)}$ is called a \emph{weight vector} of $
\mathbf{A.}$
\end{definition}

Let $q^{2}+q+1=dt.$ Then, by using the cyclic Singer group of
$PG(2,q)$, the incidence matrix of the plane $PG(2,q)$ can be
constructed as a BDC matrix $\mathbf{A}$ of the form (\ref{eq8
_block-circulant}). In this plane, we number the points
$P_{1},\ldots ,P_{q^{2}+q+1}$ and the lines $\ell _{1},\ldots
,\ell _{q^{2}+q+1}$ so that $P_{i}$ corresponds to the $i$-th
column of $\mathbf{A}$ and $\ell _{i}$ corresponds to the
$i$-th row of $\mathbf{A.}$ Denote by
$\mathbf{P}_{v}=\{P_{dv+1},\ldots ,P_{dv+d}\},$ $0\leq v\leq
t-1,$ the point set corresponding to the $ (v+1)$-th block
column of $\mathbf{A.}$ Let $\mathbf{L}_{u}=\{\ell
_{du+1},\ldots ,\ell _{du+d}\},$ $0\leq u\leq t-1,$ be the line
set corresponding to the $(u+1)$-th block row of $\mathbf{A.}$
Here and further addition and subtraction of indices are
expressed modulo $t.$

We give a development of \cite[Lemma 7.7]{BDGMP-AMC2015}.
\begin{lemma}
\label{lem bdc}Let $1\leq m\leq t-1.$ An $md$-set
$$\mathbf{P}^{(m)}=\mathbf{P }_{0}\cup \mathbf{P}_{1}\cup \ldots
\cup \mathbf{P}_{m-1}$$ corresponding to the first $md$ columns
of $\mathbf{A}$ in \eqref{eq8 _block-circulant} is a $(1,\mu
)$-saturating set $S$ in $PG(2,q)$ with
\begin{align}\label{eq8_Nvm}
&\mu =\min_{v}N_{v}^{(m)},\text{ }1\leq m\leq
v\leq t-1,\\
&N_{v}^{(m)}=\sum_{u=0}^{t-1}w_{t-u+v}\binom{w_{u}^{(m)}}{2}\geq
0,~w_{u}^{(m)}=\sum_{j=0}^{m-1}w_{t-u+j}.
\nonumber
\end{align}
Moreover, an $[md,md-3,3]_{q}2$ code $C_{S}$ corresponding to
$S$ is a $(2,\mu)$-MCF code with minimum distance $d=3$ and
$\mu$-density
\begin{align}\label{eq_mudensBDC}
\gamma_{\mu}(C_{S},2)=\frac{1}{\mu}\cdot\frac{\sum\limits_{v=m}^{t-1}N_{v}^{(m)}}{t-m}.
\end{align}
\end{lemma}

\begin{proof}
Every line of $\mathbf{L}_{u}$ is a $w_{u}^{(m)}$-secant of
$\mathbf{P} ^{(m)}.$ Let $v\geq m.$ Every point of
$\mathbf{P}_{v}$ is covered by $ w_{t-u+v}$ specimens of
$w_{u}^{(m)}$-secants of $\mathbf{P}^{(m)}$ with multiplicity
$\binom{w_{u}^{(m)}}{2}$ for $0\leq u\leq t-1.$ So, every point
of $\mathbf{P}_{v},$ $m\leq v\leq t-1,$ is covered
 by $N_{v}^{(m)}$ secants of $\mathbf{P}^{(m)}$. This implies
 \eqref{eq8_Nvm} and, together with Proposition \ref{prop optimal (1,mu) OS}, gives rise to \eqref{eq_mudensBDC}.
\end{proof}

By \cite[Th. 7.8]{BDGMP-AMC2015}, the Singer partition of
$PG(2,q)$ gives a $(1,\mu)$-OS set and the corresponding
$(1,\mu)$-APMCF code  in the following cases:
\begin{itemize}
  \item $m=t-1$ for an arbitrary weight vector;
  \item $1\le m\le t-1$ and the weight vector has the form
      $\overline{\mathbf{W}}\mathbf{ (A)}=(w_{0},w,\ldots
      ,w)$;
  \item $m=1$ and the weight vector contains exactly two
      distinct weights.
\end{itemize}

\begin{example}\label{examp 5.5 OS}
All the  $(1,\mu)$-saturating $md$-sets below are optimal by
\cite[Th. 7.8]{BDGMP-AMC2015}. The corresponding
$[md,md-3,3]_{q}2$ codes $C$ are $(1,\mu)$-APMCF with
$\gamma_{\mu}(C,2)=1$. The multiplicity $\mu$ has been
calculated by \eqref{eq8_Nvm} or by \cite[Th.
7.8]{BDGMP-AMC2015}.
\begin{itemize}
  \item [(i)] \label{ex8_Baer}Let $q$ be square. Then
      $(q^{2}+q+1)=(q+\sqrt{q}+1)(q-\sqrt{q }+1).$ Let
      $d=q+\sqrt{q}+1,$ $t=q-\sqrt{q}+1.$ There is a
      partition of $ PG(2,q)$ such that all the subsets
      $\mathbf{P}_{v}$ are disjoint Baer subplanes. We have
      $\overline{\mathbf{W}}\mathbf{(A)}=(\sqrt{q}+1,1,\ldots
      ,1)$ whence
\begin{equation*}
\mu =m\binom{\sqrt{q}+m}{2}+(q+1-m) \binom{m}{2},~~1\le m\le t-1.  \label{eq8_mu_Baer}
\end{equation*}
The case $m=1$ coincides with code of Proposition
  \ref{th_BaerSubplane}.
  \item [(ii)] Let $q=p^{4v+2},$ $p\equiv 2 \bmod 3$. Then
      by \cite[Prop. 4]{DGMP-GC},
\begin{equation*}
t=3,~d=\frac{q^{2}+q+1}{3},~w_{0}=\frac{q+2\sqrt{q}+1}{3},~w_{1}=w_{2}=w=\frac{
q-\sqrt{q}+1}{3}.
\end{equation*}
For $m=1$ we have
\begin{equation*}
\mu =w\binom{w_{0}}{2}+\binom{w}{2}(w_{0}+w)=
\frac{1}{18}(q^{3}-q\sqrt{q}-2).
\end{equation*}
  \item [(iii)]Let $q=p^{4v},$ $p\equiv 2 \bmod 3$. Then by
      \cite[Prop. 4]{DGMP-GC},
\begin{equation*}
t=3,~d=\frac{q^{2}+q+1}{3},~w_{0}=\frac{q-2\sqrt{q}+1}{3},~w_{1}=w_{2}=w=\frac{
q+\sqrt{q}+1}{3}.
\end{equation*}
For $m=1$ we have
\begin{equation*}
\mu =w\binom{w_{0}}{2}+\binom{w}{2}(w_{0}+w)=
\frac{1}{18}(q^{3}+q\sqrt{q}-1).
\end{equation*}
  \item [(iv)] Let $q=p^{2c}$. Let $t$ be a prime divisor
      of $q^{2}+q+1$. Then $ t$ divides either
      $q+\sqrt{q}+1$ or $q-\sqrt{q}+1.$ Assume that $p\pmod
      t$ is a generator of the multiplicative group of
      ${\mathbb{Z}}_{t}$. By \cite[ Prop. 6]{DGMP-GC}, in
      this case  $w_{0}=(q+1\pm (1-t) \sqrt{q})/t$,
      $w_{1}=\ldots =w_{t-1}=w=(q+1\pm \sqrt{q})/t.$ For
      $m=1$ we have
\begin{equation*}
\mu =w\binom{w_{0}}{2}+\binom{w}{2}
(w_{0}+w(t-2))=\frac{q^{3}\pm (t-2)q\sqrt{q}-t+1}{2t^{2}}.
\end{equation*}

Note that the hypothesis that $p\pmod t$ is a generator of
the multiplicative group of ${\mathbb{Z}}_{t}$ holds, e.g.
in the following cases: $q=3^{4},$ $t=7;$ $q=2^{8},$
$t=13;$ $q=5^{4},$ $t=7;$ $q=2^{12},$ $ t=19;$ $q=3^{8},$
$t=7;$ $q=2^{16},$ $t=13;$ $q=17^{4},$ $t=7;$ $p\equiv 2
\pmod t,$ $t=3.$
\item [(v)] Let $q=125$. By \cite[Tab. 1]{DGMP-GC}, there
    is the partition with $t=19$, $d=829$,
    $\overline{\mathbf{W}}\mathbf{(A)}=(4,9,9,9,9,
    4,4,9,9,4,9,9,4,4,4,4,9,4,9)$. For $m=1$ we have
    $\mu=2706$.
\end{itemize}
\end{example}

 Partitions providing
$(1,\mu)$-OS sets are not always possible. But, as rule, the
partitions provide ``good'' $(1,\mu)$-saturating sets such that
the corresponding $(1,\mu)$-MCF codes have $\mu$-density  $\gamma_{\mu}(C,2)$ of
order of magnitude  less than $1+\frac{1}{cq}, c\ge1$. In the
following we give examples of ``good'' $(1,\mu)$-saturating
sets.

\begin{example}
 Using
    the approach of \cite{DGMP-GC}, we obtain by computer
    search  partitions with $t=3$ and with three distinct
    values of $w_{i}$, see also \cite[Table 1]{DGMP-GC}. We
    take $ m=1$ and $n=d$. The values of $q,w_{i},n,\mu ,$
    and $\gamma_{\mu}(C,2)$ are given in Table \ref{Table:three weights}.
    The values of $\mu$
    and $\gamma_{\mu}(C,2)$ are obtained by \eqref{eq8_Nvm} and
    \eqref{eq_mudensBDC} respectively. In the last column  we write relation of the
    form $1+\frac{1}{cq}$ such that
    $\gamma_{\mu}(C,2)<1+\frac{1}{cq}$. One can see that $1\le
    c\le35$.
\begin{table}
\caption {Values of $\mu $ and $\mu $-density for partitions
with three distinct values of $w_{i}$ }\label{Table:three
weights}
\begin{center}
\begin{tabular}{rccc|rr|cl}
\hline
$q$ & $w_{0}$ & $w_{1}$ & $w_{2}$ & $n=d$ & $\mu$ & $\gamma_{\mu}(C,2)$ &$<$\\ \hline
7 &1&4&3&19&18& 1.0833&$1+1/q$\\
13 & 4 & 7 & 3 & 61 & 117 & 1.0256&$1+1/3q$  \\
19 & 4 & 7 & 9 & 127 & 375 & 1.0200&$1+1/2q$  \\
31 & 7 & 12 & 13 & 331 & 1656 & 1.0045&$1+1/7q$   \\
37 & 13 & 9 & 16 & 469 & 2796 & 1.0075 &$1+1/3q$  \\
43 & 19 & 13 & 12 & 631 & 4392 & 1.0024&$1+1/9q$  \\
49 & 13 & 21 & 16 & 817 & 6498 & 1.0046 &$1+1/4q$  \\
61 & 16 & 25 & 21 & 1261 & 12570 & 1.0036&$1+1/4q$   \\
67 & 28 & 19 & 21 & 1519 & 16653 & 1.0019&$1+1/7q$   \\
73 & 19 & 28 & 27 & 1801 & 21627 & 1.0008&$1+1/16q$  \\
79 & 31 & 21 & 28 & 2107 & 27363 & 1.0019&$1+1/6q$   \\
97 & 28 & 31 & 39 & 3169 & 50601 & 1.0013&$1+1/7q$   \\
103 & 28 & 37 & 39 & 3571 & 60708 & 1.0008&$1+1/11q$  \\
127 & 36 & 49 & 43 & 5419 & 113673& 1.0012&$1+1/6q$   \\
139 & 39&  49 & 52 & 6487 & 149175& 1.0007&$1+1/11q$ \\
157 & 61 & 48 & 49 & 8269 & 214848 & 1.0002&$1+1/35q$   \\
163 & 63 & 49 & 52 & 8911 & 240387 & 1.0005&$1+1/12q$   \\
\hline
\end{tabular}
\end{center}
\end{table}
\end{example}

\section{Classification of minimal and optimal
$(1,\mu)$-saturating sets in $PG(2,q)$ }\label{sec_classific}

We performed a computer based search for minimal $(1,2)$-saturating sets. The results are collected in Table
\ref{Table:1}. In the 2-nd column, the values $\overline{\ell
}(2,3,q)$ of the smallest cardinality of a 1-saturating set in
$PG(2,q)$, taken from \cite {BFMP-JG2013,DGMP-AMC2011}, are
given. The cases when $\ell (2,3,q)= \overline{\ell }(2,3,q)$
are marked by the dot \textquotedblleft $\centerdot
$\textquotedblright . In the 5-th column, we give some values
of $n$ for which minimal $(1,2)$-saturating $n$-sets in
$PG(2,q)$ exist. For $3\leq q\leq 17$, we have found the
\emph{complete spectrum} of sizes~$n$. This situation is marked
by the dot \textquotedblleft $\centerdot $ \textquotedblright .
In the 2-nd and the 5-th columns, the superscript notes the
numbers of nonequivalent sets of the corresponding size. For
$3\leq q\leq 9$, we obtain the \emph{complete classification}
of the spectrum of sizes $n$ of minimal $(1,2)$-saturating
$n$-sets in $PG(2,q)$. This situation is marked by the asterisk
*. In the 3-rd column the trivial lower bound
\eqref{eq4_triv_lower_bound} is given. Finally, the size
$q+\mu+1=q+3$ of the largest minimal (1,2)-saturating set in
$PG(2,q)$, see Proposition \ref{prop new_bounds}(ii), is written
in the 4-th column.

\begin{center}
\begin{table}
\caption{The number of nonequivalent minimal (1,2)-saturating
$n$-sets in $PG(2,q)$ and the spectrum of sizes
$n$.}\label{Table:1}
 $\renewcommand\arraystretch{1.0}
\begin{array}{c|c|c|c|c}
\hline
q & \overline{\ell }(2,3,q) & \left\lceil 2\sqrt{q}\right\rceil
\phantom{^{I^{I}}} & q+\mu+1 & \text{Spectrum of }n \\ \hline
3 & 4^{1}\centerdot & 4 & 6 & 6^{4}\centerdot \ast \\
4 & 5^{1}\centerdot & 4 & 7 & 6^{2}7^{5}\centerdot \ast \\
5 & 6^{6}\centerdot & 5 & 8 & 6^{1}7^{4}8^{18}\centerdot \ast \\
7 & 6^{3}\centerdot & 6 & 10 & 8^{13}9^{564}10^{424}\centerdot \ast \\
8 & 6^{1}\centerdot & 6 & 11 & 8^{2}9^{154}10^{3372}11^{611}\centerdot \ast
\\
9 & 6^{1}\centerdot & 6 & 12 & 8^{1}9^{57}10^{12145}11^{76749}12^{3049}
\centerdot \ast \\ \hline
11 & 7^{1}\centerdot & 7 & 14 & 10^{1348}[11-14]\centerdot \phantom{^{I^{I}}}
\\
13 & 8^{2}\centerdot & 8 & 16 & 10^{2}11^{50794}[12-16]\centerdot \\
16 & 9^{4}\centerdot & 8 & 19 & 11^{52}[12-19]\centerdot \\
17 & 10^{3640}\centerdot & 9 & 20 & [12-20]\centerdot \\ \hline
19 & 10^{36}\centerdot \phantom{^{I^{I}}} & 9 & 22 & [13-22] \\
23 & 10^{1}\centerdot & 10 & 26 & [15-26] \\
25 & 12 & 10 & 28 & [17-28] \\
27 & 12 & 11 & 30 & [17-30] \\
29 & 13 & 11 & 32 & [19-32] \\
31 & 14 & 12 & 34 & [19,21-34] \\
32 & 13 & 12 & 35 & [20-35] \\
37 & 15 & 13 & 38 & [23,26-40] \\
41 & 16 & 13 & 44 & [25,29-44] \\
43 & 16 & 14 & 46 & [25,30-46] \\
47 & 18 & 14 & 50 & [27,34-50] \\
49 & 18 & 14 & 52 & [29,34-52] \\ \hline
\end{array}
$
\end{table}
\end{center}

Using different constructions we obtained (1,2)-saturating $n$-sets in $PG(2,q)$ having several points on a conic, with sizes described in Table \ref{Table:2}.

\begin{center}
\begin{table}
\caption{Sizes of small (1,2)-saturating $n$-sets in $
PG(2,q)$.}\label{Table:2}
\begin{center}
$
\begin{array}{|c|@{\,\,}c@{\,\,}c@{\,\,}c@{\,\,}c@{\,\,}c@{\,\,}c@{\,\,}c@{\,\,}c@{\,\,}
c@{\,\,}c@{\,\,}c@{\,\,}c@{\,\,}c@{\,\,}c@{\,\,}c@{\,\,}c@{\,\,}c@{\,\,}c@{\,\,}c@{\,\,}c|}
\hline
q & 53 & 59 & 61 & 67 & 71 & 73 & 79 & 81 & 83 & 89 & 97 & 101 & 103 & 107 &
109 & 113 & 125 & 127 & 131 & 139 \\ \hline
n & 31 & 33 & 35 & 37 & 39 & 41 & 39 & 45 & 45 & 49 & 53 & 55 & 55 & 57 & 59
& 61 & 67 & 67 & 69 & 73 \\ \hline
\end{array}
$
\end{center}
\end{table}
\end{center}

The smallest cardinalities of $(1,2)$-saturating sets for each
$q$ in Table \ref{Table:1} and sizes $n$ in Table \ref{Table:2}
are smaller than $2\overline{\ell }(2,3,q)$. So, for $\mu =2$,
$r=3$, and $q$ from Tables \ref{Table:1} and~\ref{Table:2}, the
goal formulated at the end of Section~\ref{seztre} is achieved.

In Table \ref{Table:ClassificationMinmal}  a classification of
minimal $m$-saturating sets  for some values of $\mu  \leq (q +
1)\binom{q}{2}$, $q \leq 11$, is presented; the superscript
over a size indicates the number of distinct $\mu$-saturating
sets of that size (up to collineations). If the subscript is absent, then there exists at least a $(1,\mu)$-saturating set of that size.

\begin{center}
\begin{table}
\caption {Classification of minimal $(1,\mu)$-saturating sets
in $PG(2,q)$}\label{Table:ClassificationMinmal}
\begin{center}
\tabcolsep = 0.5 mm
\begin{tabular}{@{}c|c@{}|c|c|c|c|c|c|c@{}}
\hline
$\mu$&$3$&$4$&$5$&$6$&$7$&$8$&$9$&$10$\\
\hline
$q = 3$& $6^1 7^2$ &$ 7^1 8^2 $& $8^1 9^1 $& $9^3$& $9^1 10^1$ &$9^1 10^1$&$9^{1}$&$11^{1}$\\
\hline
$q = 4$& $6^1 7^2 8^7$&  $8^3 9^{10}$ & $9^6 10^7$ & $9^2 10^8 11^1$ & $10^2 11^{13}$ & $11^7 12^5$&$11^{1} 12^{16}$&$12^{5} 13^{5}$\\
\hline $q = 5$ &$8^5 96^5 $& $9^7 10^{133}$ &\begin{tabular}{c}
$10^{26}$\\$ 11^{162}$\\ \end{tabular} &\begin{tabular}{c}
$11^{121}$\\$ 12^{102}$\\ \end{tabular}&  $11^{3}
12^{361}$&\begin{tabular}{c}$12^{40}$\\$13^{437}$\\
\end{tabular}
&\begin{tabular}{c}$12^{3}13^{301}$\\$14^{28}$\\ \end{tabular}&\begin{tabular}{c}$13^{14}$\\$14^{759}$\\ \end{tabular}\\
\hline $q = 7$ & \begin{tabular}{c} $8^1 9^{10}$ \\
$10^{1506}$\\$ 11^{10014}$\end{tabular} &
 \begin{tabular}{c}$10^2$\\$ 11^{3167}$\\ $12^{67454}$\end{tabular} & \begin{tabular}{c}$11^2$\\$ 12^{11301}$\\$13^{239140}$\\
 \end{tabular}&\begin{tabular}{c}$12^{59}$\\$ 13^{83378}$\\ $14^{483925}$\end{tabular}
 &\begin{tabular}{c}$13^{430}$\\$ 14^{613065}$\\ $15^{518711}$\end{tabular}&
 \begin{tabular}{c}$14^{7418}$\\ $15^{2860573}$\\ $16$\end{tabular}&\begin{tabular}{c}$14^{8}$\\$15^{247080}$\\$16$ $17$\end{tabular}& \begin{tabular}{c}$15$\\$16$\\ $17$\end{tabular}\\\hline
 $q=8$&\begin{tabular}{c}$10^{137}$\\$ 11^{19606}$\\ $12^{40514}$\end{tabular}
 &\begin{tabular}{c}$10^1$\\$ 11^{89}$\\ $12^{63582}$\\$13^{522250}$\end{tabular}
 &\begin{tabular}{c}$10^1$\\ $12^{107}$\\$13^{239774}$\\$14^{2910961}$\end{tabular}
 &\begin{tabular}{c}$12^{4}$\\$13^{820}$\\$14^{3714769}$\\ $15$\end{tabular}&\begin{tabular}{c}$13^{3}$\\ $14^{2267544}$\\ $15$ $16$\end{tabular}&&\\\hline
 $q=9$&\begin{tabular}{c}$9^1$\\ $10^{15}$\\$11^{15351}$\\$12^{1249084}$\\$13^{596493}$\end{tabular}
 &\begin{tabular}{c}$10^1$\\
 $11^{3}$\\$12^{13809}$\\$13^{6382440}$\\$14$\end{tabular}&\begin{tabular}{c}$12^{3}$\\$13$\\ $14$\\ $15$\end{tabular}&&&&\\\hline
\hline\hline $\mu$&$11$&$12$&$13$&$14$&$15$
&$16$&$17$&$18$\\\hline
$q = 3$&$12^{1}$&$12^{1}$&-&-&-&-&-&-\\
\hline
$q=4$&$12^{1}13^{11}$&\begin{tabular}{c}$12^{1}13^{3}$\\$14^{3}$\\
\end{tabular}&$14^{10}$&$14^{4}15^{4}$&$14^{2}15^{5}$&$15^{4}16^{2}$&$15^{1}16^{4}$&$16^{5}$\\
\hline $q = 5$&\begin{tabular}{c}$13^{1}14^{198}$\\$15^{171}$\\
\end{tabular}&\begin{tabular}{c}$14^{3}$\\$15^{933}$\\
\end{tabular} &\begin{tabular}{c}$15^{159}$\\$16^{309}$\\
\end{tabular}&\begin{tabular}{c}$15^{7}$\\$16^{907}$\\
\end{tabular} &\begin{tabular}{c}$15^{2}16^{239}$\\$17^{210}$\\
\end{tabular}&\begin{tabular}{c}$16^{17}$\\$17^{741}$\\
\end{tabular}&
\begin{tabular}{c}$16^{2}17^{436}$\\$18^{104}$\\ \end{tabular}&\begin{tabular}{c}$16^{2}17^{22}$\\$18^{535}$\\ \end{tabular}\\
\hline\hline $\mu$&$19$&$20$&$21$&$22$&$23$
&$24$&$25$&$26$\\\hline
$q = 4$&$16^{2}$&$16^{1}17^{2}$&$16^{1}17^{2}$&$16^{1}18^{1}$&$16^{1}18^{1}$&$16^{1}18^{1}$&$18^{1}$&$19^{1}$\\
\hline $q = 5$&\begin{tabular}{c}$17^{2}$\\$18^{558}$\\
\end{tabular}&\begin{tabular}{c}$18^{99}$\\$19^{219}$\\ \end{tabular}&\begin{tabular}{c}$18^{7}$\\$19^{447}$\\ \end{tabular}&$19^{268}20^{6}$
&\begin{tabular}{c}$19^{18}$\\$20^{239}$\\ \end{tabular}&\begin{tabular}{c}$19^{3}$\\$20^{289}$\\ \end{tabular}&$20^{96}21^{14}$&$20^{4}21^{161}$\\
\hline\hline $\mu$&$27$&$28$&$29$&$30$&$31$
&$32$&$33$&$34$\\\hline
$q = 4$&$19^{1}$&$20^{1}$&$20^{1}$&$20^{1}$&-&-&-&-\\
\hline $q = 5$&\begin{tabular}{c}$20^{1}$\\$21^{172}$\\
\end{tabular} &\begin{tabular}{c}$20^{1}21^{39}$\\$22^{14}$\\ \end{tabular}&\begin{tabular}{c}$20^{1}21^{4}$\\$22^{77}$\\ \end{tabular}
&\begin{tabular}{c}$20^{1}21^{1}$\\$22^{88}$\\ \end{tabular}
 &\begin{tabular}{c}$22^{29}$\\$23^{8}$\\ \end{tabular}&\begin{tabular}{c}$22^{5}$\\$23^{32}$\\ \end{tabular}
 &\begin{tabular}{c}$22^{2}$\\$23^{38}$\\ \end{tabular}&\begin{tabular}{c}$22^{1}23^{18}$\\$24^{6}$\\ \end{tabular}\\
\hline\hline $\mu$&$35$&$36$&$37$&$38$&$39$
&$40$&$41$&$42$\\\hline
 $q = 5$&$23^{4}24^{16}$&$24^{22}$&$24^{16}25^{1}$&$24^{4}25^{7}$&$24^{1}25^{10}$&$25^{12}$&$25^{6}26^{1}$&$25^{2}26^{3}$\\
\hline\hline $\mu$&$43$&$44$&$45$&$46$&$47$
&$48$&$49$&$50$\\\hline
 $q = 5$&$25^{1}26^{4}$&$25^{1}26^{4}$&$25^{1}26^{1}$&$25^{1}27^{2}$&$25^{1}27^{2}$&$25^{1}27^{2}$&$25^{1}28^{1}$&$25^{1}28^{1}$\\
\hline\hline
$\mu$&$51$&$52$&$53$&$54$&$55$&$56$&$57\,|\,58$&$59\,|\,60$\\\hline
 $q = 5$&$27^{1}28^{1}$&$28^{2}$&$28^{1}$&$29^{1}$&$29^{1}$&$29^{1}$&$30^{1}|30^{1}$&$30^{1}|30^{1}$\\\hline
\end{tabular}
\end{center}
\end{table}
\end{center}

\begin{remark}\label{rem_minim=set}
By property (M3) of Definition \ref{def_rho_mu_sat}, see also
(M3) in Section \ref{seztre}, a $(1,\mu)$-saturating set is
also a $(1,\mu-k)$-saturating set for $1\le k\le \mu-1$.
Moreover, a minimal $(1,\mu)$-saturating set can also be a
minimal $(1,\mu-k)$-saturating set for $k=1,2,\ldots,\delta$,
$\delta\ge1$. This happens when removing any point from this
set we obtain a $(1,\mu-\delta-1)$-saturating set.
 For example, let $q=3$ and
consider a line $\ell$. Then $S=PG(2,3)\setminus \ell$ is a
$(1,9)$-saturating set and  removing any point from $S$ we
obtain a $(1,4)$-saturating set. So, $S$ is a minimal $(1,9)$-,
$(1,8)$-, $(1,7)$-, $(1,6)$-, and $(1,5)$-saturating set. This
example and many other such situations are written in Table
\ref{Table:ClassificationMinmal}.
\end{remark}

In Table \ref{Table:ClassificationOptimal} we give the
classification of optimal $(1,\mu)$-saturating sets
($(1,\mu)$-OS)in $PG(2,q)$. For such sets $S$, every point of
$PG(2,q)\setminus S$ is covered exactly $\mu$ times. An entry of the form $n^{t}_{\mu}$ means that there
exist $t$ projectively distinct optimal $(1,\mu)$-saturating
sets with size $n$. Entries provided by Propositions
\ref{th7_(n,s)arc} -- \ref{th_BaerSubplane} and \ref{th_optim},  Example \ref{examp 5.5 OS}, and Corollary \ref{Corollary} are written in bold font.
\begin{center}
\begin{table}
\caption {Classification of optimal $(1,\mu)$-saturating
$n$-sets in
$PG(2,q)$}\label{Table:ClassificationOptimal}
\begin{center}
\tabcolsep = 0.5 mm
\begin{tabular}{|c||c|}
\hline
$q$& $n^{t}_{\mu} \vphantom{H^{H^{H}}_{H_{H}}}$\\\hline
$3$&$\mathbf{7^{1}_{3}~7^{1}_{4}~9^{1}_{6}~10^{1}_{8}}~9^{1}_{9}~
\mathbf{10^{1}_{9}~11^{1}_{10}~12^{1}_{12}}~\vphantom{H^{H^{H}}_{H_{H}}}$ \\\hline
$4$&\begin{tabular}{c}$\mathbf{6^{1}_{3}~7^{1}_{3}~9^{1}_{4}}~9^{1}_{5}~
\mathbf{9^{1}_{6}}~11^{1}_{9}~\mathbf{12^{1}_{9}~12^{2}_{10}~13^{2}_{12}~14^{1}_{15}~15^{1}_{15}}~15^{1}_{16}$\\
$~15^{1}_{17}~ \mathbf{16^{1}_{18}~17^{1}_{21}~16^{1}_{24}~17^{1}_{24}~18^{1}_{24}~18^{1}_{25}~19^{1}_{27}~20^{1}_{30}} $
\end{tabular}
 \\\hline
$5$&\begin{tabular}{c}$
\mathbf{11^{1}_{5}}~14^{1}_{11}~\mathbf{15^{1}_{12}~16^{1}_{15}}~16^{1}_{18}~19^{1}_{22}~\mathbf{19^{2}_{23}}
~21^{1}_{27}~\mathbf{21^{2}_{30}}~22^{2}_{31}~22^{1}_{32}~23^{1}_{35}$\\
$~\mathbf{25^{1}_{40}}~25^{1}_{41}~25^{1}_{42}~\mathbf{26^{1}_{44}~27^{1}_{48}}~25^{1}_{50}
~\mathbf{26^{1}_{50}~27^{1}_{51}~28^{1}_{52}~28^{1}_{53}~29^{1}_{56}~30^{1}_{60}}$
\end{tabular}
\\\hline
\end{tabular}
\end{center}
\end{table}
\end{center}
\begin{observation}\label{observ}
For $q =3$, the (1,3)-OS of size 7 is 2 concurrent lines; the
(1,4)-OS of size 7 is contained in 3 concurrent lines. For $q
=4$ the (1,3)-OS of size 6 is the hyperoval, the (1,4)-OS of
size 7 is a complete 3-arc, cf. with Proposition
\ref{th7_(n,s)arc}.

In all these 4 cases, the points external to the $(1,\mu)$-OS,
say $S$, form a unique orbit of the stabilizer group of $S$.
\end{observation}

In the following we give some constructions of optimal
$(1,\mu)$-OS in $PG(2,q)$ providing  many entries in Table
\ref{Table:ClassificationOptimal}.

\begin{theorem}\label{th_optim} The following sets $\mathcal{S}$ are
$(1,\mu)$-OS in $PG(2,q)$.
\begin{itemize}
   \item [(i)] The set $\mathcal{S}$ is the union of $L$
       concurrent lines, $2\le L\le q$. It holds that
       $$|\mathcal{S}|=1+Lq, ~\mu=\binom{L}{2}q.$$
   \item [(ii)] The set $\mathcal{S}$ is the union of $q$
       concurrent lines and $b$ other points on the
       $(q+1)$-th one, $1\le b\le q-1$. It holds that
       $$|\mathcal{S}|=1+q^{2}+b,
       ~\mu=\binom{b+1}{2}+\binom{q}{2}q.$$
   \item [(iii)] The set $\mathcal{S}$ is a triangle. It
       holds that
       $$|\mathcal{S}|=3q,~\mu=3+(q-2)\binom{3}{2}=3(q-1).$$
   \item [(iv)] The set $\mathcal{S}=PG(2,q)\setminus T$
       where $T$ is a vertex-less triangle. It
       holds that
       $$|\mathcal{S}|=q^{2}-2q+4,~\mu=1+\binom{q}{2}+(q-1)\binom{q-2}{2}.$$

 \end{itemize}
\end{theorem}
\begin{proof} The sizes of $\mathcal{S}$ are obvious. Let $P$ be a point
of $PG(2,q)\setminus \mathcal{S}$. Let $G$ be the intersection
point of concurrent lines.
\begin{itemize}
   \item [(i)]  Every line through $P$ distinct from $PG$
       intersects  $\mathcal{S}$ in $L$ points.
   \item [(ii)] Every line through $P$ distinct from $PG$
       intersects  $\mathcal{S}$ in $q$ points.   The line
       $PG$ provides is a $(b+1)$-secant of $\mathcal{S}$.
   \item [(iii)] Three lines through $P$ and one of
       vertices of the triangle  are bisecants of
       $\mathcal{S}$, whereas every other line is a
       3-secant.
        \item [(iv)] This follows directly from Proposition \ref{TheoremComplement}.
\end{itemize}
\end{proof}


\medskip
{\it E-mail address: }dbartoli@cage.ugent.be\\
\indent{\it E-mail address: }adav@iitp.ru\\
\indent{\it E-mail address: }massimo.giulietti@unipg.it\\
\indent{\it E-mail address: }stefano.marcugini@unipg.it\\
\indent{\it E-mail address: }fernanda.pambianco@unipg.it

\begin{thebibliography}{99}




\bibitem{BFMP-JG2013} D. Bartoli, G. Faina, S. Marcugini, and
    F. Pambianco, On the minimum size of complete arcs and minimal
    saturating sets in projective planes, \emph{J. Geom.}, \textbf{104}
    (2013) 409--419.

\bibitem{BDGMP-ACCT2012}
\newblock D. Bartoli, A. A. Davydov, M. Giulietti, S. Marcugini and F. Pambianco,
\newblock Multiple coverings of the farthest-off points and multiple saturating sets in projective spaces,
\newblock in \emph{Proc. XIII Int. Workshop Algebr. Combin. Coding Theory, ACCT2012}, Pomoria, Bulgaria, 2012, 53--59.

\bibitem{BDGMP-AMC2015}
\newblock D. Bartoli, A. A. Davydov, M. Giulietti, S. Marcugini and F. Pambianco,
\newblock Multiple coverings of the farthest-off points with small density from projective
geometry,
\newblock \emph{Adv. Math. Commun.}, \textbf{9} (2015), 63--85.



\bibitem{Bo-Sz-Ti} (MR2181039) [10.1016/j.disc.2004.12.015]
\newblock E. Boros, T. Sz\H{o}nyi and K. Tichler,
\newblock On defining sets for projective planes,
\newblock \emph{Discrete Math.}, \textbf{303} (2005), 17--31.

\bibitem{BrLyWi-Handbook} R. A. Brualdi, S. Litsyn, and V. S.
    Pless,
    Covering radius, in ``Handbook of Coding Theory" (eds. V.S.
    Pless, W.C. Huffman and R.A. Brualdi), Elsevier, Amsterdam,
    The Netherlands, (1998), 755--826.




\bibitem{Coh} (MR1453577)
\newblock G. Cohen, I. Honkala, S. Litsyn and A. Lobstein,
\newblock \emph{Covering Codes},
\newblock North-Holland, Amsterdam, 1997.

\bibitem{Dav95} (MR1385598) [10.1109/18.476339]
\newblock A. A. Davydov,
\newblock Constructions and families of covering codes and saturated sets of points in projective geometry,
\newblock \emph{IEEE Trans. Inform. Theory}, \textbf{41} (1995), 2071--2080.

\bibitem{DFGMP-ConfigParam}
\newblock A. A. Davydov, G. Faina, M. Giulietti, S. Marcugini and F. Pambianco,
\newblock On constructions and parameters of symmetric configurations $v_k$,
\newblock \emph{Des. Codes Cryptogr.}, to appear, DOI 10.1007/s10623-015-0070-x.

\bibitem{DGMP-Petersb2009}
\newblock A. A. Davydov, M. Giulietti, S. Marcugini and F. Pambianco,
\newblock On the spectrum of possible parameters of symmetric configurations,
\newblock in \emph{Proc. XII Int. Symp. Probl. Redundancy Inform. Control Systems}, Saint-Petersburg, Russia, 2009, 59--64.

\bibitem{DGMP-2010}
\newblock A. A. Davydov, M. Giulietti, S. Marcugini and F. Pambianco,
\newblock New inductive constructions of complete caps in $PG(N,q)$, $q$ even,
\newblock  \emph{J. Comb. Des.}, {\bf 18} (2010), 176--201.

\bibitem{DGMP-AMC2011} (MR2770105) [10.3934/amc.2011.5.119]
\newblock A. A. Davydov, M. Giulietti, S. Marcugini and F. Pambianco,
\newblock Linear nonbinary covering codes and saturating sets in projective spaces,
\newblock \emph{Adv. Math. Commun.}, \textbf{5} (2011), 119--147.

\bibitem{DGMP-GC} (MR3027596) [10.1007/s00373-011-1103-5]
\newblock A. A. Davydov, M. Giulietti, S. Marcugini and F. Pambianco,
\newblock Some combinatorial aspects of constructing bipartite-graph codes,
\newblock \emph{Graphs Combin.}, \textbf{29} (2013), 187--212.



\bibitem{GEJC} (MR2365974)
\newblock M. Giulietti,
\newblock On small dense sets in Galois planes,
\newblock \emph{Electr. J. Combin.}, \textbf{14} (2007), \#75.

\bibitem{GJCD} (MR2343872) [10.1002/jcd.20131]
\newblock M. Giulietti,
\newblock Small complete caps in $PG(N,q), q$ even,
\newblock \emph{J. Combin. Des.}, \textbf{15} (2007), 420--436.

\bibitem{GiulBritCombConf} (MR3156928)
\newblock M. Giulietti,
\newblock The geometry of covering codes: small complete caps and saturating sets in Galois spaces,
\newblock in \emph{Surveys in Combinatorics}, Cambridge Univ. Press, 2013, 51--90.

\bibitem{GiPa} (MR2317157) [10.1109/TIT.2007.894688]
\newblock M. Giulietti and F. Pasticci,
\newblock Quasi-perfect linear codes with minimum distance 4, \emph{IEEE}
\newblock  \emph{Trans. Inform. Theory}, \textbf{53} (2007), 1928--1935.

\bibitem{GT2004} (MR2069044)
\newblock M. Giulietti and F. Torres,
\newblock On dense sets related to plane algebraic curves,
\newblock \emph{Ars Combinatoria}, \textbf{72} (2004), 33--40.


\bibitem{MMCCFF2} (MR1228543) [10.1007/BF01388486]
\newblock H. O. H\"am\"al\"ainen, I. S. Honkala, M. K. Kaikkonen and S. N. Litsyn,
\newblock Bounds for binary multiple covering codes,
\newblock \emph{Des. Codes Cryptogr.}, \textbf{3} (1993), 251--275.

\bibitem{HHLO-Siam} (MR1329506) [10.1137/S0895480193252100]
\newblock H. O. H\"am\"al\"ainen, I. S. Honkala, S. N. Litsyn and P. R. J. \"Osterg\aa rd,
\newblock Bounds for binary codes that are multiple coverings of the farthest-off points,
\newblock \emph{SIAM J. Discrete Math.}, \textbf{8} (1995), 196--207.

\bibitem{HHLO-AMM} (MR1349869) [10.2307/2974552]
\newblock H. H\"am\"al\"ainen, I. Honkala, S. Litsyn and P. \"Osterg\aa rd,
\newblock Football pools -- a game for mathematicians,
\newblock \emph{Amer. Math. Monthly}, \textbf{102} (1995), 579--588.

\bibitem{MR1082845} (MR1082845) [10.1016/0097-3165(91)90024-B]
\newblock H. O. H\"am\"al\"ainen and S. Rankinen,
\newblock Upper bounds for football pool problems and mixed covering codes,
\newblock \emph{J. Combin. Theory Ser. A}, \textbf{56} (1991), 84--95.


\bibitem{Hirs1998} (MR1612570)
\newblock J. W. P. Hirschfeld,
\newblock \emph{Projective Geometries over Finite Fields}, $2^{nd}$ edition,
\newblock Oxford University Press, Oxford, 1998.

\bibitem{MMCCFF5} (MR1263752) [10.1016/0012-365X(94)90164-3]
\newblock I. S. Honkala,
\newblock On the normality of multiple covering codes,
\newblock \emph{Discrete Math.}, \textbf{125} (1994), 229--239.

\bibitem{HonkLits1996} (MR1386954)
\newblock I. Honkala and S. Litsyn,
\newblock Generalizations of the covering radius problem in coding theory,
\newblock \emph{Bull. Inst. Combin.}, \textbf{17} (1996), 39--46.




\bibitem{MR1441669} (MR1441669) [10.1023/A:1008228721072]
\newblock P. R. J. \"Osterg\aa rd and H. O. H\"am\"al\"ainen,
\newblock A new table of binary/ternary mixed covering codes,
\newblock \emph{Des. Codes Cryptogr.}, \textbf{11} (1997), 151--178.

\bibitem{BPFMsatu}
\newblock F. Pambianco, D. Bartoli, G. Faina and S. Marcugini,
\newblock Classification of the smallest minimal 1-saturating sets in $PG(2,q)$, $q\le 23$,
\newblock \emph{Electron. Notes Discrete Math.}, \textbf{40} (2013) 229--233.

\bibitem{PDBGM-ExtendAbstr}
\newblock F. Pambianco, A. A. Davydov, D. Bartoli, M. Giulietti and S. Marcugini,
\newblock A note on multiple coverings of the farthest-off points,
\newblock \emph{Electron. Notes Discrete Math.}, \textbf{40} (2013) 289--293.

\bibitem{Quistorff} (MR1801437)
\newblock J. Quistorff,
\newblock On Codes with given minimum distance and covering radius,
\newblock Beitr\"age Algebra Geom., \textbf{41} (2000) 469--478.

\bibitem{MMCCFF1} (MR1865545)
\newblock J. Quistorff,
\newblock Correction: On codes with given minimum distance and covering radius,
\newblock Beitr\"age Algebra Geom., \textbf{42} (2001), 601--611.


\bibitem{TamasRoma}
\newblock T. Sz\H{o}nyi,
\newblock Complete arcs in Galois planes: a survey,
\newblock in \emph{Quaderni del Seminario di Geometrie Combinatorie 94}, Universit\`a degli Studi di Roma ``La Sapienza'', Roma, 1989.

\bibitem{WCL1991}
\newblock G. J. M. van Wee, G. D. Cohen and S. N. Litsyn,
\newblock A note on perfect multiple coverings of the Hamming space,
\newblock \emph{IEEE Trans. Inform. Theory}, \textbf{37} (1991), 678--682.

\end{thebibliography}
\end{document}